\documentclass[11pt]{amsart}

\usepackage{amssymb}
\usepackage{geometry}
\usepackage[dvipsnames]{xcolor}
\usepackage{hyperref}
\usepackage[foot]{amsaddr}
\hypersetup{
    colorlinks=true,
    linkcolor=MidnightBlue,
    citecolor=MidnightBlue,
    urlcolor=MidnightBlue,
    pdftitle={Trembling-Hand Perfection in Mean Field Games},
    pdfauthor={Luciano Campi, Luca Di Persio, Lucrezia Zorzi}
}
\numberwithin{equation}{section}

\newtheorem{theorem}{Theorem}[section]
\newtheorem{lemma}{Lemma}[section]
\newtheorem{proposition}{Proposition}[section]
\newtheorem{remark}{Remark}[section]
\newtheorem{assumption}{Assumption}[section]
\newtheorem{definition}{Definition}[section]

\title[Trembling-Hand Perfection in Mean Field Games]
{Trembling-Hand Perfection in Mean Field Games}

\date{September 4, 2026}

\author[Campi]{Luciano Campi $^{\ast,1}$}
\email{\href{mailto:luciano.campi@unimi.it}{$^{1}$luciano.campi@unimi.it}}

\author[Di Persio]{Luca Di Persio $^{\dagger,2}$}
\email{\href{mailto:luca.dipersio@univr.it}{$^{2}$luca.dipersio@univr.it}}

\author[Zorzi]{Lucrezia Zorzi $^{\dagger,3}$}
\email{\href{mailto:lucrezia.zorzi@univr.it}{$^{3}$lucrezia.zorzi@univr.it}}

\address{$^{\ast}$Department of Mathematics ``Federigo Enriques'', University of Milan, Milan, Italy.}

\address{$^{\dagger}$Department of Computer Science, University of Verona, Verona, Italy.}

\begin{document}

\begin{abstract}
We introduce an admissible trembling-hand refinement for stochastic mean field games formulated through relaxed controlled martingale problems. A relaxed MFG equilibrium is perfect if it is the common Wasserstein limit of two sequences of joint control-state laws. The first consists of full-support population perturbations admissible for the state laws they generate; the second consists of exact optimal responses to those perturbed state laws. This admissibility condition prevents the control and state coordinates of a perturbation from being varied independently. Under the standing continuity, growth, and coercivity assumptions, we prove existence without
requiring compact controls or bounded coefficients. The proof combines a perturbed
fixed-point argument in the compact-bounded case with truncation, uniform
moment estimates, compactness, and stability of controlled martingale
problems. A one-dimensional model with exactly two relaxed mean field game
equilibria shows that the refinement is genuinely selective: only one
equilibrium is trembling-hand perfect.
\end{abstract}

\maketitle

{\noindent \small \textbf{Keywords:} mean field games; trembling-hand perfection; equilibrium selection; relaxed controls; controlled martingale problems.}

\smallskip

{\noindent \small \textbf{AMS 2020:} 91A16, 91A11, 93E20.}

\smallskip

\section{Introduction}

Mean field games (MFGs), introduced independently by Lasry and Lions \cite{lasry2007mean} and Huang, Malhamé and Caines \cite{huang2006large}, provide a framework for the analysis of strategic
interactions in symmetric stochastic differential games for a large population of players. In the
limiting formulation, a representative agent optimizes against a prescribed
flow of population distributions, and an equilibrium is obtained when the law
generated by an optimal response coincides with that flow. This fixed point
relation between individual optimality and collective behavior is the
defining feature of a mean field game; see
\cite{cardaliaguet2018course} and \cite{carmona2017probabilistic,
carmona2017probabilistic2} for general treatments.

In the absence of structural conditions ensuring uniqueness, mean field games
may admit multiple equilibria. Uniqueness holds under the classical Lasry-Lions monotonicity condition \cite{lasry2007mean} and, in certain
settings, for sufficiently small time horizons; see \cite{bardi2019non}.
Non-uniqueness is, however, a genuine feature of many mean field game models;
see, among others,
\cite{bardi2019non,bayraktar2020non,briani2018stable,cecchin2019convergence,delarue2019restoring,delarue2020selection,tchuendom2018uniqueness}.
This motivates the search for equilibrium refinements capable of selecting
among multiple solutions.

A natural refinement is Selten's notion of trembling-hand perfection, which requires an equilibrium to be robust to arbitrarily small mistakes modeled
through completely mixed perturbations 
\cite{selten1975reexamination, osborne1994course, van1991stability}.
A literal extension to mean field games of the finite-game characterization of Selten's notion would, however, be too restrictive if it
required the candidate equilibrium control itself to remain optimal against
every perturbed population flow. Indeed, in the present infinite-dimensional
setting, best responses may depend on the perturbation and converge to the
equilibrium without ever coinciding with it. Related
difficulties have led to several extensions of trembling-hand perfection
for games with infinite strategy spaces, including Simon's notion of local perfection in \cite{simon1987local} and the notions studied by Simon and Stinchcombe \cite{simon1995equilibrium}. More recently, Flesch et al.\ \cite{flesch2025general} proposed a two-sequence formulation in which both
the completely mixed perturbations and their exact optimal responses may
vary, provided that both sequences converge to the same equilibrium. This formulation is particularly suited to our setting, where optimal responses to perturbed population flows may vary with the perturbation and converge to the candidate equilibrium without ever coinciding with it. In finite normal-form games, the formulation of Flesch et al.\ coincides with Selten's original notion.

We adapt this two-sequence structure to the relaxed mean field game
framework developed by Lacker \cite{lacker2015mean}, in which controls and state
processes are represented by their joint laws. A relaxed MFG equilibrium
$\widehat P$ is called trembling-hand perfect (THP) if there are two sequences
$(P^j)_{j\geq1}$ and $(Q^j)_{j\geq1}$ converging to $\widehat P$ in the
$p$-Wasserstein topology. The law $P^j$ represents the perturbed population
behavior: its relaxed-control marginal has full support, and its control
and state coordinates jointly satisfy the controlled martingale problem in
the environment generated by its own state law. The law $Q^j$ is an exact
optimal response to that perturbed environment.

The admissibility requirement on $P^j$ is essential. Full support alone
could be imposed artificially by coupling an arbitrary full-support control
law with the state law of the unperturbed equilibrium. Such a construction
would leave the population environment unchanged and would make every
relaxed MFG equilibrium satisfy the remaining requirements of the
definition. The admissibility requirement excludes this possibility by linking the
perturbed control coordinate to the resulting state law through the
controlled dynamics. Thus the two principal requirements have distinct
roles: full support is the analogue of complete mixing, whereas
admissibility ensures that the perturbation represents feasible population
behavior.

A related trembling-hand approach is developed by Graber \cite{graber2025trembling} for a class of deterministic mean field games
associated with transport equations. His criterion perturbs the dynamics
through stochastic noise and then takes a vanishing-noise limit. Our construction
instead leaves the underlying dynamics unchanged and introduces trembles
directly at the level of relaxed population controls.

Our main result establishes the existence of a THP equilibrium under the
general assumptions of \cite{lacker2015mean}. We first prove existence in the
setting with compact controls and bounded coefficients using a perturbed fixed
point argument. We then extend the result to the general case by combining
coefficient truncation, compactness arguments, and uniform moment estimates.
Finally, we present a one-dimensional model with two relaxed mean field game
equilibria and show that exactly one is a THP equilibrium. The example also
shows that the selection mechanism depends on the interaction between the full
support condition and the geometry of the control set. 

The paper is organized as follows.
Section~\ref{section: relaxed framework} introduces the relaxed control and
controlled martingale problem framework.
Section~\ref{section: thp} defines trembling-hand perfection and states the
main existence results.
Sections~\ref{section: compact-bounded existence} and
\ref{section: general existence} establish existence, respectively, under
Assumption~\ref{ass_compact-bounded} and in the general setting.
Section~\ref{section: selection example} presents a selection example, and
Section~\ref{section: conclusion} concludes.
Technical results on
full-support perturbations, compactness, stability, and approximation of
admissible deviations are collected in the appendices.

\section{Relaxed controlled martingale problems}
\label{section: relaxed framework}

This section introduces the relaxed control and controlled martingale problem framework used throughout the paper. The representative agent problem is formulated in terms of joint laws of relaxed controls and state trajectories, separating admissibility from optimality. This distinction will be essential when perturbing population behavior while allowing the corresponding optimal response to vary.

\subsection{Notation and model data}
\label{subsection: notation data}

Fix a finite time horizon $T>0$, a state dimension $d\in\mathbb{N}$, and a noise dimension $m\in\mathbb{N}$. Throughout, $|\cdot|$ denotes the Euclidean norm on finite-dimensional Euclidean spaces; in particular, on matrix spaces it denotes the
Frobenius norm. For $k\in\mathbb{N}$, let $\mathcal{C}^{k}:=C([0,T];\mathbb{R}^{k}),$ endowed with the supremum norm
$$
\|x\|_{t}:=\sup_{0\leq s\leq t}|x_s|,
\qquad
t\in[0,T],\quad x\in\mathcal{C}^{k}.
$$
Let $(E,d_E)$ be a complete separable metric space. For $r\geq1$, denote by $\mathcal{P}^{r}(E)$ the set of Borel probability measures $\nu$ on $E$ such that $\int_E d^r_E(x,x_0)\,\nu(dx)<\infty$ for some $x_0\in E$. We equip $\mathcal{P}^{r}(E)$ with the $r$-Wasserstein distance
$$
W_{r,E}(\nu,\nu') := \left(\inf_{\pi\in\Pi(\nu,\nu')} 
\int_{E\times E}d^r_E(x,y)\,\pi(dx,dy)\right)^{1/r},
$$
where $\Pi(\nu,\nu')$ denotes the set of couplings of $\nu$ and $\nu'$. We write $W_r$ when the underlying space is clear.

For $\nu\in\mathcal{P}^{r}(\mathbb{R}^{k})$, set
$$
m_r(\nu) := \left(\int_{\mathbb{R}^{k}}|z|^r\,\nu(dz)\right)^{1/r}.
$$
For $\mu\in\mathcal{P}^{r}(\mathcal{C}^{k})$, define
$$
\|\mu\|_{t,r}^{r} := \int_{\mathcal{C}^{k}}\|x\|_{t}^{r}\,\mu(dx),
\qquad t\in[0,T].
$$
The time-$t$ marginal of $\mu$ is denoted by
$$
\mu_t:=\mu\circ e_t^{-1},
\qquad
e_t(x):=x_t.
$$
Fix exponents $p,p' \geq 1$ and $p_\sigma \geq 0$.
Let $A$ denote the control space. The mean field game is specified by
an initial distribution $\lambda \in \mathcal{P}^p(\mathbb{R}^d)$ and measurable functions
$$
\begin{aligned}
b&:[0,T]\times\mathbb{R}^{d}
\times\mathcal{P}^{p}(\mathbb{R}^{d})\times A
\rightarrow\mathbb{R}^{d},\\
\sigma&:[0,T]\times\mathbb{R}^{d}
\times\mathcal{P}^{p}(\mathbb{R}^{d})\times A
\rightarrow\mathbb{R}^{d\times m},\\
f&:[0,T]\times\mathbb{R}^{d}
\times\mathcal{P}^{p}(\mathbb{R}^{d})\times A
\rightarrow\mathbb{R},\\
g&:\mathbb{R}^{d}\times\mathcal{P}^{p}(\mathbb{R}^{d})
\rightarrow\mathbb{R}.
\end{aligned}
$$

The functions $b$ and $\sigma$ are the drift and diffusion coefficients, while $f$ and $g$ are the running and terminal rewards.

For intuition, fix a population law $\mu\in\mathcal{P}^{p}(\mathcal{C}^{d})$ and consider a strict control $\alpha$. On a suitable filtered probability space supporting an $m$-dimensional Brownian motion $W$, the representative state formally satisfies
$$
dX_t = b(t,X_t,\mu_t,\alpha_t)\,dt + \sigma(t,X_t,\mu_t,\alpha_t)\,dW_t, 
\qquad
X_0\sim\lambda,
$$
and the representative agent maximizes
$$
\mathbb{E}\left[\int_0^T f(t,X_t,\mu_t,\alpha_t)\,dt + g(X_T,\mu_T) \right].
$$
A mean field game equilibrium requires the prescribed population law to coincide with the law generated by an optimal control. The relaxed controlled martingale problem formulation below is the one used throughout the paper.

\subsection{Standing assumptions}
\label{subsection: standing assumptions}

We work under the following standard hypotheses for relaxed controlled martingale problems, adapted from \cite[Assumption~(A)]{lacker2015mean}.

\begin{assumption}[Standing assumptions]
\label{ass_standing}
The following conditions hold.
\begin{enumerate}
\item[\textnormal{(A1)}]
The functions $b$, $\sigma$, and $f$ are measurable in $t$ and continuous in $(x,\mu,a)\in
\mathbb{R}^{d}\times
\mathcal{P}^{p}(\mathbb{R}^{d})\times A$. The function $g$ is continuous in $(x,\mu)$.

\item[\textnormal{(A2)}]
There exists $c_1>0$ such that, for every $(t,\mu,a)\in
[0,T]\times\mathcal{P}^{p}(\mathbb{R}^{d})\times A$
and $x,y\in\mathbb{R}^{d}$,
$$
|b(t,x,\mu,a)-b(t,y,\mu,a)| + |\sigma(t,x,\mu,a)-\sigma(t,y,\mu,a)| \leq c_1|x-y|.
$$
Moreover,
$$
|b(t,x,\mu,a)| \leq c_1\left(1+|x|+m_p(\mu)+|a|\right),
$$
and
$$
|\sigma\sigma^\top(t,x,\mu,a)| \leq
c_1\left(1+|x|^{p_\sigma}+m^{p_\sigma}_p(\mu)+|a|^{p_\sigma}\right).
$$

\item[\textnormal{(A3)}]
There exist constants $c_2,c_3>0$ such that, for every $(t,x,\mu,a)\in
[0,T]\times\mathbb{R}^{d}
\times\mathcal{P}^{p}(\mathbb{R}^{d})\times A$,
$$
|g(x,\mu)| \leq
c_2\left(1+|x|^p+m^p_p(\mu)\right),
$$
$$
-c_2\left(1+|x|^p+m^p_p(\mu)+|a|^{p'}\right) \leq
f(t,x,\mu,a) \leq
c_2\left(1+|x|^p+m^p_p(\mu)\right) -c_3|a|^{p'}.
$$

\item[\textnormal{(A4)}]
The control space $A$ is a nonempty closed subset of a finite-dimensional Euclidean space.

\item[\textnormal{(A5)}]
The initial distribution satisfies $\lambda\in\mathcal{P}^{p'}(\mathbb{R}^{d})$, and
$$
p'>p\geq1\vee p_\sigma,
\qquad
p_\sigma\in[0,2].
$$
\end{enumerate}
\end{assumption}

The coercive upper bound in \textnormal{(A3)} yields $p'$-moment estimates for optimal controls. The condition $p'>p$ provides the uniform integrability needed for compactness in $\mathcal{P}^{p}$, while $p\geq p_\sigma$ and $p_\sigma\leq2$ control the diffusion terms.

\begin{assumption}[Compact controls and bounded coefficients]
\label{ass_compact-bounded}
The control space $A$ is compact, and the functions $b$ and $\sigma$ are bounded.
\end{assumption}

Assumption~\ref{ass_compact-bounded} is used only in the first stage of the existence proof. The general result is obtained under Assumption~\ref{ass_standing} by truncating the control space and the coefficients and then passing to the limit.

\subsection{Relaxed controls and the canonical space}
\label{subsection: relaxed controls}

\begin{definition}[Relaxed controls]
\label{definition: relaxed controls}
Let $\mathcal{V}[A]$ be the set of finite Borel measures $q$ on $[0,T]\times A$ such that
$$
q([s,t]\times A)=t-s,
\qquad
0\leq s<t\leq T,
$$
and
$$
\int_{[0,T]\times A}|a|^p\,q(dt,da)<\infty.
$$
An element $q\in\mathcal{V}[A]$ is called a relaxed control.
\end{definition}

Every $q\in\mathcal{V}[A]$ admits a disintegration
$q(dt,da)=dt\,q_t(da)$, where $t\mapsto q_t\in\mathcal{P}(A)$ is measurable and uniquely determined for Lebesgue-a.e.\ $t$. 
A strict control $t\mapsto\alpha_t\in A$ is identified with
$q(dt,da)=dt\,\delta_{\alpha_t}(da)$. We equip $[0,T]\times A$ with the product metric
$$
d_{[0,T]\times A}\big((t,a),(s,a')\big) := |t-s| + |a-a'|.
$$
We equip $\mathcal{V}[A]$ with the metric
$$
d_{\mathcal{V}[A]}(q,q') :=
W_{p,[0,T]\times A}\left(\frac{q}{T},\frac{q'}{T}\right).
$$
With this metric, $\mathcal{V}[A]$ is complete and separable; if $A$ is compact, then $\mathcal{V}[A]$ is compact.

The canonical space is
$\Omega[A]:=\mathcal{V}[A]\times\mathcal{C}^{d}$,
equipped with the product metric
$$
d_{\Omega[A]}\bigl((q,x),(q',x')\bigr) :=
\left(d_{\mathcal{V}[A]}(q,q')^p+\|x-x'\|_T^p\right)^{1/p}.
$$
Its canonical coordinates are
$$
\Lambda(q,x):=q,
\qquad
X(q,x):=x.
$$
We write $\Omega$ and $\mathcal{V}$ when the dependence on $A$ is unambiguous.
For $t\in[0,T]$, define
$$
\mathcal{F}^{\Lambda}_t :=
\sigma\left(\Lambda(C):C\in\mathcal{B}([0,t]\times A)\right), 
\qquad
\mathcal{F}^{X}_t:=\sigma(X_s:0\leq s\leq t),
$$
and $\mathcal{F}_t := \sigma\left(\mathcal{F}^{\Lambda}_t\cup\mathcal{F}^{X}_t\right)$.

By \cite[Lemma~3.2]{lacker2015mean}, we fix an
$(\mathcal{F}^{\Lambda}_t)$-predictable version of the
disintegration of the canonical relaxed control, denoted by
$(t,q)\mapsto\Lambda_t(q)$, such that $q(dt,da)=dt\,\Lambda_t(q)(da)$ for every $q\in\mathcal{V}[A]$.
For $r\geq1$, we use the notation
$$
|\Lambda_t|^r := \int_A |a|^r\,\Lambda_t(da).
$$
The coordinate projections are $1$-Lipschitz. Hence the maps
$P\mapsto P\circ\Lambda^{-1}$ and $P\mapsto P\circ X^{-1}$
are continuous from $\mathcal{P}^{p}(\Omega[A])$ into $\mathcal{P}^{p}(\mathcal{V}[A])$ and $\mathcal{P}^{p}(\mathcal{C}^{d})$, respectively.

\begin{lemma}[Continuity of time marginals]
\label{lemma: time marginal continuity}
For $\mu,\nu\in\mathcal{P}^{p}(\mathcal{C}^{d})$,
$$
\sup_{t\in[0,T]}W_p(\mu_t,\nu_t) \leq W_{p,\mathcal{C}^{d}}(\mu,\nu).
$$
Consequently, convergence in $\mathcal{P}^{p}(\mathcal{C}^{d})$ implies
uniform convergence of the time marginals in $W_p$.
\end{lemma}

\begin{proof}
For every coupling $\pi$ of $\mu$ and $\nu$, the image of $\pi$ under
$(x,y)\mapsto(x_t,y_t)$ is a coupling of $\mu_t$ and $\nu_t$. Hence
$$
W_p^p(\mu_t,\nu_t) \leq 
\int_{\mathcal C^d\times\mathcal C^d}|x_t-y_t|^p\,\pi(dx,dy) \leq 
\int_{\mathcal C^d\times\mathcal C^d}\|x-y\|_T^p\,\pi(dx,dy).
$$
Taking the infimum over $\pi$ and then the supremum over $t$ proves the claim. 
\end{proof}

\subsection{Controlled martingale problems}
\label{subsection: martingale problem}

Fix an external population law $\mu\in\mathcal{P}^{p}(\mathcal{C}^{d})$. The representative agent treats the flow $(\mu_t)_{t\in[0,T]}$ as exogenous.

For $\varphi\in C_c^\infty(\mathbb{R}^{d})$, define
$$
\mathcal{L}^{\mu,a}_t\varphi(x) :=
b(t,x,\mu_t,a)^\top D\varphi(x) + 
\frac12 \operatorname{Tr}\left[\sigma\sigma^\top(t,x,\mu_t,a)D^2\varphi(x)\right].
$$
For $(q,x)\in\Omega[A]$, set
$$
M^{\mu,\varphi}_t(q,x) := \varphi(x_t) -
\int_0^t\int_A \mathcal{L}^{\mu,a}_s\varphi(x_s) q_s(da)\,ds.
$$

\begin{definition}[Admissible law]
\label{definition: admissible law}
For $\mu\in\mathcal{P}^{p}(\mathcal{C}^{d})$, let $\mathcal{R}(\mu)$ be the set of probability measures $P\in\mathcal{P}(\Omega[A])$ such that:
\begin{enumerate}
\item $P\circ X_0^{-1}=\lambda$;
\item $\mathbb{E}^{P} \left[ \int_0^T|\Lambda_t|^p\,dt \right] <\infty;$
\item for every $\varphi\in C_c^\infty(\mathbb{R}^{d})$, the process $(M^{\mu,\varphi}_t)_{t\in[0,T]}$ is a $P$-martingale with respect to $(\mathcal{F}_t)_{t\in[0,T]}$.
\end{enumerate}
\end{definition}

Thus $\mathcal{R}(\mu)$ is the set of admissible joint laws of controls and states when the external environment is $\mu$.

\begin{proposition}[Martingale measure representation]
\label{proposition: martingale measure representation}
For $\mu\in\mathcal{P}^{p}(\mathcal{C}^{d})$,
$\mathcal{R}(\mu)$ is precisely the set of laws $P'\circ(\Lambda,X)^{-1},$ where:
\begin{enumerate}
\item $(\Omega',\mathcal{F}',(\mathcal{F}'_t)_{t\in[0,T]},P')$ is a filtered probability space carrying a predictable
$\mathcal{P}(A)$-valued process $(\Lambda_t)_{t\in[0,T]}$, a continuous
$(\mathcal{F}'_t)$-adapted process $X$, and $m$ orthogonal
$(\mathcal{F}'_t)$-martingale measures $N=(N^1,\ldots,N^m)$ on $A\times[0,T]$, each with intensity $\Lambda_t(da)\,dt$;
\item $P'\circ X_0^{-1}=\lambda;$
\item $\mathbb{E}^{P'}\left[\int_0^T|\Lambda_t|^p\,dt\right]<\infty;$
\item the state equation holds:
$$
dX_t = \int_A b(t,X_t,\mu_t,a)\Lambda_t(da)\,dt + 
\int_A \sigma(t,X_t,\mu_t,a)N(da,dt).
$$
\end{enumerate}
Here the process $(\Lambda_t)_{t\in[0,T]}$ is identified with the random relaxed control $\Lambda(dt,da)=dt\,\Lambda_t(da).$
\end{proposition}

\begin{proof}
This is \cite[Theorem~IV.2]{karoui1990martingale}. 
\end{proof}

\begin{lemma}[State estimate]
\label{lemma: state estimate}
Suppose Assumption~\ref{ass_standing} holds and fix
$\gamma\in[p,p']$. There exists a constant $c_4>0$, depending only on
$\gamma$, $T$, $c_1$, and the $p'$-moment of $\lambda$, such that, for
every $\mu\in\mathcal{P}^{p}(\mathcal{C}^{d})$ and
$P\in\mathcal{R}(\mu)$,
$$
\mathbb{E}^{P}\left[\|X\|_T^\gamma\right] \leq
c_4 \left(1+\|\mu\|_{T,\gamma}^{\gamma} +
\mathbb{E}^{P} \left[\int_0^T|\Lambda_t|^\gamma\,dt \right] \right),
$$
whenever the right-hand side is finite. In particular, $P\in\mathcal{P}^{p}(\Omega[A]).$
Moreover, if $P\circ X^{-1}=\mu$, then
$$
\|\mu\|_{T,\gamma}^{\gamma} =
\mathbb{E}^{P}\left[\|X\|_T^\gamma\right] \leq
c_4 \left(1+ \mathbb{E}^{P} \left[ \int_0^T|\Lambda_t|^\gamma\,dt \right] \right),
$$
whenever the control moment on the right-hand side is finite.
\end{lemma}

\begin{proof}
This is \cite[Lemma~4.3]{lacker2015mean}. 
\end{proof}

\subsection{Reward functional, optimal laws, and MFG equilibria}
\label{subsection: optimal laws}

For $\mu\in\mathcal{P}^{p}(\mathcal{C}^{d})$, define
$$
\Gamma^\mu(q,x) := \int_0^T\int_A f(t,x_t,\mu_t,a)q_t(da)\,dt +
g(x_T,\mu_T).
$$
For $P\in\mathcal{P}^{p}(\Omega[A])$, set
$$
J(\mu,P) := \int_{\Omega[A]}\Gamma^\mu(q,x)\,P(dq,dx).
$$
The upper bound in Assumption~\ref{ass_standing}\textnormal{(A3)} ensures that the positive part of $\Gamma^\mu$ is $P$-integrable. Hence $J(\mu,P)\in\mathbb{R}\cup\{-\infty\}$ is well defined.

\begin{lemma}[Higher moments of finite-value admissible laws]
\label{lemma: finite value higher moments}
Let $\mu\in\mathcal P^p(\mathcal C^d)$ and
$S\in\mathcal R(\mu)$. If $J(\mu,S)>-\infty,$ then
$$
\mathbb E^S \left[\int_0^T|\Lambda_t|^{p'}\,dt \right] <\infty.
$$
If, in addition, $\|\mu\|_{T,p'}<\infty$, then $\mathbb E^S\bigl[\|X\|_T^{p'}\bigr]<\infty.$
\end{lemma}

\begin{proof}
Since $S\in\mathcal R(\mu)$, we have $\mathbb E^S[\int_0^T|\Lambda_t|^p\,dt] <\infty.$
Moreover, $\mu\in\mathcal P^p(\mathcal C^d)$, and therefore the state
estimate of Lemma~\ref{lemma: state estimate}, applied with $\gamma=p$,
gives $\mathbb E^S\bigl[\|X\|_T^p\bigr]<\infty$.
Using the coercive upper bound for $f$ and the growth bound for $g$, we
obtain
$$
J(\mu,S) \leq
C\left(1+\|\mu\|_{T,p}^p +\mathbb E^S\bigl[\|X\|_T^p\bigr] \right)
- c_3\mathbb E^S \left[\int_0^T|\Lambda_t|^{p'}\,dt \right].
$$
The first term on the right-hand side is finite. Since
$J(\mu,S)>-\infty$, it follows that
$$
\mathbb E^S\left[\int_0^T|\Lambda_t|^{p'}\,dt\right]<\infty.
$$
If, in addition, $\|\mu\|_{T,p'}<\infty$, the state estimate of
Lemma~\ref{lemma: state estimate}, now applied with $\gamma=p'$, yields the second assertion. 
\end{proof}

\begin{proposition}[Upper semicontinuity of the reward]
\label{proposition: J upper semicontinuity}
Under Assumption~\ref{ass_standing}, the map
$$
J: \mathcal{P}^{p}(\mathcal{C}^{d}) \times \mathcal{P}^{p}(\Omega[A])
\rightarrow \mathbb{R}\cup\{-\infty\}
$$
is upper semicontinuous.
\end{proposition}

\begin{proof}
This is \cite[Lemma~4.5]{lacker2015mean}. 
\end{proof}

\begin{definition}[Optimal admissible laws]
\label{def:optimal admissible laws}
For $\mu\in\mathcal P^p(\mathcal C^d)$, define
$$
\mathcal R^*(\mu) :=
\operatorname*{arg\,max}_{P\in\mathcal R(\mu)} J(\mu,P).
$$
\end{definition}

At this stage, $\mathcal{R}^{*}(\mu)$ may be empty. Its nonemptiness and the continuity properties needed for the fixed point argument will be established under the hypotheses of the corresponding existence result.

For $P\in\mathcal{P}^{p}(\Omega[A])$, define
$\mu^{P}:=P\circ X^{-1} \in\mathcal{P}^{p}(\mathcal{C}^{d})$ and
$\mu^{P}_t:=P\circ X_t^{-1}$.

\begin{definition}[Relaxed MFG equilibrium]
\label{definition: relaxed mfg equilibrium}
A law $\widehat{P}\in\mathcal{P}^{p}(\Omega[A])$ is a relaxed mean field game equilibrium if
$\widehat{P} \in \mathcal{R}^{*}(\mu^{\widehat{P}})$.
\end{definition}

Thus a relaxed MFG equilibrium is admissible in the environment generated by its own state law and maximizes the representative agent's reward in that environment.

\section{Trembling-hand perfection}
\label{section: thp}

This section introduces the equilibrium refinement studied in the paper and states the main existence results. The definition is formulated on the canonical space introduced in Section~\ref{section: relaxed framework}. Its central feature is that the perturbed population behavior is represented by an admissible joint law of controls and states, while the representative agent is allowed to choose a possibly different exact optimal response to the population flow generated by that law.

\subsection{Perturbations with full support}
\label{subsection: full support perturbations}

We first specify the meaning of full support in the relaxed control setting. Since $\mathcal{V}[A]$ is a Polish space, the support of a probability measure $\eta\in\mathcal{P}(\mathcal{V}[A])$ is the closed set
$$
\operatorname{supp}(\eta) := \left\{q\in\mathcal{V}[A]:\eta(O)>0
\text{ for every open neighborhood }O\text{ of }q \right\}.
$$

\begin{definition}[Probability measures with full support]
\label{definition: full-support law}
A probability measure $\eta\in\mathcal{P}(\mathcal{V}[A])$ has full support if
$\operatorname{supp}(\eta)=\mathcal{V}[A]$.
Equivalently, $\eta(O)>0$ for every nonempty open set $O\subset\mathcal{V}[A]$.
\end{definition}

The condition is imposed on a probability law over the space of entire relaxed controls. It should not be confused with a pointwise requirement of the form $\operatorname{supp}(q_t)=A$ for Lebesgue-a.e. $t$.
No such pointwise requirement is used below.

Given $\eta \in \mathcal{P}(\mathcal{V}[A])$, define its barycentric relaxed control
$\bar{\eta} \in \mathcal{V}[A]$ by
$$
\bar{\eta}(B) := \int_{\mathcal{V}[A]} q(B)\,\eta(dq),
\qquad
B \in \mathcal{B}([0,T]\times A).
$$
Then $\bar{\eta}(dt,da)=dt\,\bar{\eta}_t(da)$ for a measurable kernel $(\bar{\eta}_t)_{t\le T}$.

\begin{proposition}[Full support of the barycentric control]
\label{proposition: barycentric full support}
If $\operatorname{supp}(\eta)=\mathcal V[A]$, then
$$
\operatorname{supp}(\bar \eta)=[0,T]\times A \quad\text{and}\quad
\operatorname{supp}(\bar \eta_t)=A\quad\text{for Lebesgue-a.e. }t.
$$
\end{proposition}

The proposition does not assert that an $\eta$-distributed relaxed control
has full support on $A$ realization-wise. It identifies the precise
population-level consequence of full support on $\mathcal V[A]$. Its proof is
given in Appendix~\ref{appendix: full-support perturbations}.

Probability measures with full support exist because $\mathcal{V}[A]$ is separable. In the compact case, if $(q^k)_{k\geq1}$ is a countable dense subset of $\mathcal{V}[A]$, then
$$
\eta:=\sum_{k=1}^{\infty}2^{-k}\delta_{q^k}
$$
has full support. For the general existence proof, the perturbation law must also satisfy the stronger integrability condition
\begin{equation}
\label{eq: perturbation law pprime moment}
\int_{\mathcal{V}[A]}
\left(
\int_0^T\int_A |a|^{p'}q_t(da)\,dt
\right)
\eta(dq)
<\infty.
\end{equation}
Proposition~\ref{proposition: global full-support law appendix}
constructs a probability measure
$\eta\in\mathcal{P}^{p}(\mathcal{V}[A])$
with full support and property~\eqref{eq: perturbation law pprime moment}.
Thus the compact case uses the elementary atomic law introduced above,
whereas the general existence proof uses the stronger construction from
Appendix~\ref{appendix: full-support perturbations}.

\subsection{The refinement}
\label{subsection: refinement definition}

We now define the equilibrium concept.
\begin{definition}[Trembling-hand perfect equilibrium]
\label{definition: thp}
Let $\widehat{P}\in\mathcal{P}^{p}(\Omega[A])$ be a relaxed MFG equilibrium. We call $\widehat{P}$ a trembling-hand perfect (THP) equilibrium if there exist two sequences
$(P^j)_{j\geq1}$ and $(Q^j)_{j\geq1}$ in $\mathcal{P}^{p}(\Omega[A])$ such that:
\begin{enumerate}
\item $P^j\rightarrow\widehat{P},
\,
Q^j\rightarrow\widehat{P}$
in $\mathcal{P}^{p}(\Omega[A])$;
\item for every $j\geq1$, the relaxed control marginal $P^j\circ\Lambda^{-1}$ has full support on $\mathcal{V}[A]$;
\item for every $j\geq1$,
$$
P^j\in\mathcal{R}(\mu^{P^j}),
\qquad
Q^j\in\mathcal{R}^{*}(\mu^{P^j}).
$$
\end{enumerate}
\end{definition}
The two sequences have distinct roles. The law $P^j$ describes a perturbation of the population behavior. The condition $P^j\in\mathcal{R}(\mu^{P^j})$ requires its control and state coordinates to satisfy the controlled martingale problem in the environment generated by its own state law. Thus the perturbation remains admissible and is not merely a device for producing an external population flow. The law $Q^j$, by contrast, is an exact optimal response to the perturbed environment $\mu^{P^j}$. Allowing $Q^j$ to vary with $j$ is essential in the present infinite-dimensional strategy space, where best responses need not remain fixed under perturbations.

The convergence in Definition~\ref{definition: thp} is convergence in the $p$-Wasserstein space $\mathcal{P}^{p}(\Omega[A])$, not merely weak convergence. By continuity of the canonical projections, this convergence also holds for the corresponding control and state marginals. Full support is required only for the perturbing relaxed control marginals; it need not be inherited by the limiting equilibrium.

\begin{remark}[Role of admissibility]
\label{remark: admissibility thp}
The admissibility requirement on $P^j$ prevents the definition from becoming vacuous. To see this, let $\widehat{P}$ be any relaxed MFG equilibrium, let $\eta$ be a probability measure with full support satisfying~\eqref{eq: perturbation law pprime moment}, and consider
$$
R:=\eta\otimes\mu^{\widehat{P}} \quad\text{on }\Omega[A].
$$
Then we have
$R\circ\Lambda^{-1}=\eta$ and $R\circ X^{-1}=\mu^{\widehat{P}}$,
but in general $R\notin\mathcal{R}(\mu^{\widehat{P}})$, because its control and state coordinates need not satisfy the controlled martingale problem jointly. If $\varepsilon_j\downarrow0$ and one were to set 
$$ 
P^j:=(1-\varepsilon_j)\widehat P+\varepsilon_jR, 
\qquad 
Q^j:=\widehat P, 
$$ then the standard mixture coupling gives 
$$ 
W_{p,\Omega[A]}^p(P^j,\widehat P) \le 
\varepsilon_j W_{p,\Omega[A]}^p(R,\widehat P) \rightarrow0, 
$$ 
while $Q^j=\widehat P$ for every $j$. Hence both sequences converge to $\widehat P$. Moreover, the relaxed control marginal of $P^j$ has full support, and
$$
\mu^{P^j}=\mu^{\widehat{P}},
\qquad
Q^j\in\mathcal{R}^{*}(\mu^{P^j}).
$$
Thus, without the condition $P^j\in\mathcal{R}(\mu^{P^j})$, every relaxed MFG equilibrium would satisfy the remaining requirements through an artificial coupling of independent control and state coordinates.

The full-support and admissibility requirements therefore serve logically distinct purposes. The condition $\operatorname{supp}(P^j\circ\Lambda^{-1})=\mathcal V[A]$ is a topological complete-mixing condition on population strategies, whereas $P^j\in\mathcal R(\mu^{P^j})$ links the state marginal to the control coordinate through the dynamics. The selection example in Section~\ref{section: selection example} uses both: full support gives positive probability to controls with positive aggregate drift, while admissibility converts this into a strict perturbation of the terminal population mean.
\end{remark}

\subsection{Main existence results}
\label{subsection: main results strategy}

We first state the result in the case covered by Assumption~\ref{ass_compact-bounded}.

\begin{theorem}[Existence under Assumption~\ref{ass_compact-bounded}]
\label{thm: compact-bounded thm}
Suppose that Assumptions~\ref{ass_standing} and \ref{ass_compact-bounded} hold. Then there exists at least one THP equilibrium.
\end{theorem}

The main theorem removes the compactness and boundedness restrictions.

\begin{theorem}[General existence]
\label{thm: general thp}
Suppose that Assumption~\ref{ass_standing} holds. Then there exists at least one THP equilibrium.
\end{theorem}

Theorems~\ref{thm: compact-bounded thm} and~\ref{thm: general thp}
are proved in Sections~\ref{section: compact-bounded existence}
and~\ref{section: general existence}, respectively.
The general proof uses the perturbation law with an integrable
$p'$-moment and its truncated approximations constructed in
Appendix~\ref{appendix: full-support perturbations}.

\section{Existence with compact controls and bounded coefficients}
\label{section: compact-bounded existence}

In this section we prove Theorem~\ref{thm: compact-bounded thm}. Throughout the section, Assumptions~\ref{ass_standing} and \ref{ass_compact-bounded} are in force. We first collect the compactness and continuity properties inherited from the relaxed controlled martingale problem framework. We then introduce the constrained admissible law correspondence used to impose perturbations with full support. The proof is completed by a perturbed fixed point argument and a vanishing perturbation limit.

\subsection{Admissible and optimal law correspondences}
\label{subsection: compact correspondences}

We begin with the properties of the admissible law correspondence
$$
\mathcal{R}:\mathcal{P}^{p}(\mathcal{C}^{d}) \rightrightarrows
\mathcal{P}^{p}(\Omega[A]),
$$
given in Definition \ref{definition: admissible law}.

\begin{proposition}[Properties of the admissible law correspondence]
\label{proposition: R compact-bounded}
Under Assumptions \ref{ass_standing} and \ref{ass_compact-bounded}, the following statements hold.
\begin{enumerate}
\item For every $\mu\in\mathcal{P}^{p}(\mathcal{C}^{d})$, the set $\mathcal{R}(\mu)$ is nonempty, compact, and convex in $\mathcal{P}^{p}(\Omega[A])$.
\item The range
$$
\operatorname{Ran}(\mathcal{R}) := 
\left\{ P\in\mathcal{P}^{p}(\Omega[A]): P\in\mathcal{R}(\mu)
\text{ for some } \mu\in\mathcal{P}^{p}(\mathcal{C}^{d})\right\}
$$
is relatively compact in $\mathcal{P}^{p}(\Omega[A])$.
\item The correspondence $\mathcal{R}$ is continuous: it is both upper and lower hemicontinuous with respect to the $p$-Wasserstein topologies.
\end{enumerate}
\end{proposition}

\begin{proof}
Relative compactness of $\operatorname{Ran}(\mathcal{R})$ and continuity of
$\mathcal{R}$ follow from \cite[Lemma~4.4]{lacker2015mean}. Nonemptiness
follows by fixing a strict control and solving the corresponding state
equation in the prescribed environment. Convexity follows because the initial
law, moment and martingale constraints are affine in the joint law. For fixed
$\mu$, the value $\mathcal R(\mu)$ is closed by the closed-graph part of the
same lemma. A closed subset of the relatively compact range is compact, which
proves the compactness of each value. 
\end{proof}

The reward functional is continuous under Assumption~\ref{ass_compact-bounded}.

\begin{proposition}[Continuity of the reward]
\label{proposition: J continuity compact}
Under Assumptions~\ref{ass_standing} and \ref{ass_compact-bounded}, the map $J$ is continuous.
\end{proposition}

\begin{proof}
The conclusion follows from \cite[Lemma~4.5]{lacker2015mean}. 
\end{proof}

\begin{proposition}[Properties of the optimal law correspondence]
\label{proposition: Rstar compact-bounded}
Under Assumptions~\ref{ass_standing} and \ref{ass_compact-bounded}, the correspondence
$$
\mathcal{R}^{*}: \mathcal{P}^{p}(\mathcal{C}^{d}) \rightrightarrows
\mathcal{P}^{p}(\Omega[A]),
$$
given in Definition \ref{def:optimal admissible laws}, has nonempty compact convex values and is upper hemicontinuous. Its range is relatively compact, and its graph is closed.
\end{proposition}

\begin{proof}
For each $\mu$, Propositions~\ref{proposition: R compact-bounded} and \ref{proposition: J continuity compact} ensure that the maximum is attained, so $\mathcal{R}^{*}(\mu)$ is nonempty and compact. Convexity follows from convexity of $\mathcal{R}(\mu)$ and linearity of $J(\mu,\cdot)$. Berge's maximum theorem \cite[Theorem~17.31]{aliprantis2006infinite} gives upper hemicontinuity. Since $\mathcal R^*(\mu)$ is compact, and hence closed,
for every $\mu$, the upper hemicontinuity of $\mathcal R^*$ implies that
its graph is closed. Finally,
$$
\operatorname{Ran}(\mathcal{R}^{*}) \subset \operatorname{Ran}(\mathcal{R}),
$$
and therefore Proposition \ref{proposition: R compact-bounded}(2) implies that
$\operatorname{Ran}(\mathcal R^*)$ is relatively compact. 
\end{proof}

We next construct a compact convex set containing all state laws generated by admissible controls.

\begin{lemma}[Invariant compact set of state laws]
\label{lemma: compact invariant state set}
Under Assumptions~\ref{ass_standing} and \ref{ass_compact-bounded}, there exists a compact convex set
$\mathcal{Q} \subset \mathcal{P}^{p}(\mathcal{C}^{d})$
such that
$P\circ X^{-1}\in\mathcal{Q}$
for every $\mu\in\mathcal{P}^{p}(\mathcal{C}^{d})$ and every $P\in\mathcal{R}(\mu)$.
\end{lemma}

\begin{proof}
By boundedness of $b$ and $\sigma$, a standard
Burkholder--Davis--Gundy estimate gives a constant $C>0$ such that
$$
\mathbb{E}^{P} \left[ \|X\|_T^{p'} \right] \leq
C \left(1+\int_{\mathbb{R}^{d}} |x|^{p'}\,\lambda(dx)\right)
$$
for every $\mu\in\mathcal{P}^{p}(\mathcal{C}^{d})$ and every
$P\in\mathcal{R}(\mu)$. In particular,
$$
M := \sup \left\{ \mathbb{E}^{P} \left[\|X\|_T^{p'}\right]:
\mu\in\mathcal{P}^{p}(\mathcal{C}^{d}), P\in\mathcal{R}(\mu)\right\} <\infty.
$$
By Assumption~\ref{ass_compact-bounded}, for each
$\varphi\in C_c^\infty(\mathbb{R}^{d})$ there exists
$C_\varphi>0$ such that
$\left| \mathcal{L}^{\mu,a}_t\varphi(x) \right| \leq C_\varphi$,
for every $(t,x,\mu,a)\in[0,T]\times\mathbb{R}^{d} \times\mathcal{P}^{p}(\mathbb{R}^{d})\times A.$
Moreover, $C_\varphi$ depends only on the bounds of
$D\varphi$ and $D^2\varphi$.

Let $\mathcal{Q}$ denote the set of probability measures
$\nu$ on $\mathcal{C}^{d}$ satisfying the following conditions:
\begin{enumerate}
\item $\nu\circ X_0^{-1} = \lambda;$
\item $\|\nu\|_{T,p'}^{p'} \leq M;$
\item for every nonnegative
$\varphi\in C_c^\infty(\mathbb{R}^{d})$, the process $\varphi(X_t)+C_\varphi t$ is a $\nu$-submartingale with respect to the canonical
filtration $(\mathcal{F}^{X}_t)_{t\in[0,T]}$.
\end{enumerate}

It is clear that $\mathcal{Q}$ is convex. We next show that it
contains the state marginal of every admissible law. Fix
$\mu\in\mathcal{P}^{p}(\mathcal{C}^{d})$ and $P\in\mathcal{R}(\mu)$. For each nonnegative $\varphi\in C_c^\infty(\mathbb{R}^{d})$,
$$
\varphi(X_t)+C_\varphi t = M^{\mu,\varphi}_t +
\int_0^t\int_A \left(\mathcal{L}^{\mu,a}_s\varphi(X_s) +
C_\varphi \right) \Lambda_s(da)\,ds.
$$
The first term is a $P$-martingale, while the second term is
nondecreasing. Hence
$\varphi(X_t)+C_\varphi t$ is a $P$-submartingale and therefore
a $P\circ X^{-1}$-submartingale with respect to the canonical
state filtration. The definition of $M$ also gives
$\|P\circ X^{-1}\|_{T,p'}^{p'} \leq M$.
Thus $P\circ X^{-1} \in \mathcal{Q}$.

By \cite[Theorem~1.4.6]{stroock2007multidimensional},
the submartingale condition implies that $\mathcal{Q}$ is tight in
$\mathcal{P}(\mathcal{C}^{d})$. The uniform $p'$-moment bound
in condition~\textnormal{(2)}, together with $p'>p$, implies
uniform $p$-integrability. Hence $\mathcal{Q}$ is relatively
compact in $\mathcal{P}^{p}(\mathcal{C}^{d})$.

Finally, $\mathcal{Q}$ is closed in $\mathcal{P}^p(C^d)$. Indeed, let
$\nu_n\in \mathcal{Q}$ and $\nu_n\to\nu$ in $\mathcal{P}^p(C^d)$. The initial-law
constraint passes to the limit by continuity of the evaluation map
$x\mapsto x_0$, while the moment bound is preserved by lower
semicontinuity. To verify the submartingale condition, fix
$0\leq s<t\leq T$, a nonnegative $\phi\in C_c^\infty(\mathbb{R}^d)$,
and a bounded nonnegative continuous $\mathcal{F}_s^X$-measurable
function $h$. Since $\nu_n\in \mathcal{Q}$,
$$
\int h\bigl(\phi(X_t)-\phi(X_s)+C_\phi(t-s)\bigr)\,d\nu_n\geq0.
$$
The integrand is bounded and continuous, so the same inequality holds
under $\nu$ by weak convergence. A standard monotone-class argument
then yields the submartingale property under $\nu$. Hence $\nu\in \mathcal{Q}$,
and therefore $\mathcal{Q}$ is compact.  
\end{proof}

\subsection{Admissible laws with prescribed control marginal}
\label{subsection: prescribed control marginal}

We use the atomic probability measure with full support introduced in Subsection~\ref{subsection: full support perturbations}. Thus, for a countable dense subset $(q^k)_{k\geq1}$ of $\mathcal{V}[A]$, set $\eta := \sum_{k=1}^{\infty}2^{-k}\delta_{q^k}.$
For each $\nu\in\mathcal{P}^{p}(\mathcal{C}^{d})$, define
$$
\mathcal{R}_{\eta}(\nu) := 
\left\{R\in\mathcal{R}(\nu):R\circ\Lambda^{-1}=\eta\right\}.
$$

\begin{proposition}[The prescribed marginal correspondence]
\label{proposition: Reta properties}
For every $\nu\in\mathcal{P}^{p}(\mathcal{C}^{d})$, the set $\mathcal{R}_{\eta}(\nu)$ is nonempty, compact, and convex. Moreover, the correspondence
$$
\mathcal{R}_{\eta}: \mathcal{P}^{p}(\mathcal{C}^{d}) \rightrightarrows
\mathcal{P}^{p}(\Omega[A])
$$
is upper hemicontinuous.
\end{proposition}

\begin{proof}
For each $k\geq1$, fix a measurable disintegration
$$
q^k(dt,da)=dt\,q_t^k(da).
$$
Since $q^k$ is deterministic, the construction of
\cite[Proposition~IV.1]{karoui1990martingale} provides a filtered
probability space carrying $m$ orthogonal martingale measures
$N^k=(N^{k,1},\ldots,N^{k,m})$ on $A\times[0,T]$, each with intensity
$q_t^k(da)\,dt$, together with an $\mathcal F_0$-measurable initial state of
law $\lambda$, independent of the martingale measures. The Lipschitz and
growth assumptions then yield pathwise uniqueness and a strong solution
$X^{k,\nu}$ of
$$
dX_t^{k,\nu} = \int_A b(t,X_t^{k,\nu},\nu_t,a)q_t^k(da)\,dt +
\int_A \sigma(t,X_t^{k,\nu},\nu_t,a)N^k(da,dt),
\qquad
X_0^{k,\nu}\sim\lambda.
$$
Set $R^{k,\nu} := \operatorname{Law}(q^k,X^{k,\nu}).$
By Proposition~\ref{proposition: martingale measure representation},
$$
R^{k,\nu}\in\mathcal{R}(\nu),
\qquad
R^{k,\nu}\circ\Lambda^{-1} = \delta_{q^k}.
$$
Define
$$
R^\nu := \sum_{k=1}^{\infty}2^{-k}R^{k,\nu}.
$$
The initial law constraint, the control integrability condition in Definition~\ref{definition: admissible law}(2), and the
martingale identities defining $\mathcal{R}(\nu)$ are affine in the law.
Since $A$ is compact, the required control moment is uniformly bounded.
Therefore $R^\nu\in\mathcal{R}(\nu).$
Moreover,
$$
R^\nu\circ\Lambda^{-1} = \sum_{k=1}^{\infty} 2^{-k}\delta_{q^k} = \eta.
$$
Thus $R^\nu\in\mathcal{R}_\eta(\nu),$ which proves nonemptiness.

Convexity follows from convexity of $\mathcal{R}(\nu)$ and the affine
marginal constraint. Since the map $R\mapsto R\circ\Lambda^{-1}$ is continuous, $\mathcal{R}_\eta(\nu)$ is a closed subset of the compact
set $\mathcal{R}(\nu)$ and is therefore compact.
Finally, let $\nu^k\rightarrow\nu$ in $\mathcal{P}^{p}(\mathcal{C}^{d})$, and let $R^k\in\mathcal{R}_\eta(\nu^k).$
Relative compactness of $\operatorname{Ran}(\mathcal{R})$ and the closed
graph property of $\mathcal{R}$ imply that every limit point $R$ of
$(R^k)$ belongs to $\mathcal{R}(\nu)$. Continuity of the relaxed control
projection gives
$R\circ\Lambda^{-1} = \eta$.
Thus $R\in\mathcal{R}_\eta(\nu),$ which proves upper hemicontinuity.  
\end{proof}

\subsection{The perturbed fixed point problem}
\label{subsection: perturbed fixed point compact}

Fix $\varepsilon\in(0,1)$. For $\nu\in\mathcal{Q}$, define
$$
\mathcal{G}^{\varepsilon}(\nu) :=
\bigl\{(1-\varepsilon)Q\circ X^{-1} + \varepsilon R\circ X^{-1}:
Q\in\mathcal{R}^{*}(\nu),\, R\in\mathcal{R}_{\eta}(\nu)\bigr\}.
$$

\begin{proposition}[Perturbed fixed point]
\label{proposition: perturbed fixed point existence}
For every $\varepsilon\in(0,1)$, there exist
$\nu^{\varepsilon}\in\mathcal{Q}$,
$Q^{\varepsilon}\in\mathcal{R}^{*}(\nu^{\varepsilon})$,
$R^{\varepsilon}\in\mathcal{R}_{\eta}(\nu^{\varepsilon})$
such that
$$
\nu^{\varepsilon} =
(1-\varepsilon)Q^{\varepsilon}\circ X^{-1} + \varepsilon R^{\varepsilon}\circ X^{-1}.
$$
\end{proposition}

\begin{proof}
Propositions~\ref{proposition: Rstar compact-bounded} and \ref{proposition: Reta properties} imply that $\mathcal{G}^{\varepsilon}$ has nonempty compact convex values. Moreover, if
$$
\theta = (1-\varepsilon)Q\circ X^{-1} + \varepsilon R\circ X^{-1}
\in \mathcal{G}^{\varepsilon}(\nu),
$$
then $Q,R\in\mathcal{R}(\nu)$. Lemma~\ref{lemma: compact invariant state set} and convexity of $\mathcal{Q}$ therefore give $\theta\in\mathcal{Q}$.

To prove upper hemicontinuity, let $\nu^k\to\nu$ in $\mathcal{Q}$ and $\theta^k \in
\mathcal{G}^{\varepsilon}(\nu^k),$ with $\theta^k\rightarrow\theta.$
Choose
$Q^k\in\mathcal{R}^{*}(\nu^k)$ and $R^k\in\mathcal{R}_{\eta}(\nu^k)$
such that
$$
\theta^k = (1-\varepsilon)Q^k\circ X^{-1} + \varepsilon R^k\circ X^{-1}.
$$
Relative compactness of $\operatorname{Ran}(\mathcal{R})$ permits a common subsequence, not relabelled, such that
$Q^k\rightarrow Q$ and $R^k\rightarrow R$
in $\mathcal{P}^{p}(\Omega[A])$. Upper hemicontinuity of $\mathcal{R}^{*}$ and $\mathcal{R}_{\eta}$ gives $Q\in\mathcal{R}^{*}(\nu)$ and $R\in\mathcal{R}_{\eta}(\nu)$.
Passing to the limit in the convex mixture yields
$$
\theta = (1-\varepsilon)Q\circ X^{-1} + \varepsilon R\circ X^{-1},
$$
and hence $\theta\in\mathcal{G}^{\varepsilon}(\nu)$.

Let $\mathcal{M}(\mathcal{C}^{d})$ denote the vector space of finite signed measures on $\mathcal{C}^{d}$, endowed with the weak topology $\sigma\bigl(\mathcal{M}(\mathcal{C}^{d}),C_b(\mathcal{C}^{d})\bigr).$
This is a locally convex Hausdorff topological vector space. Since $\mathcal{Q}$ is compact in $\mathcal{P}^{p}(\mathcal{C}^{d})$, the $p$-Wasserstein and weak topologies coincide on $\mathcal{Q}$. Thus $\mathcal{Q}$ is a compact convex subset of $\mathcal{M}(\mathcal{C}^{d})$ in the weak topology. The Kakutani--Fan--Glicksberg fixed point theorem \cite[Corollary~17.55]{aliprantis2006infinite}, gives $\nu^{\varepsilon} \in \mathcal{G}^{\varepsilon}(\nu^{\varepsilon})$,
and the stated representation follows from the definition of $\mathcal{G}^{\varepsilon}$.  
\end{proof}

\subsection{Vanishing perturbations}
\label{subsection: vanishing perturbation compact}

We now complete the proof under Assumption~\ref{ass_compact-bounded}.

\begin{proof}[Proof of Theorem~\ref{thm: compact-bounded thm}]
For each $\varepsilon\in(0,1)$, let $\nu^{\varepsilon}, Q^{\varepsilon}, R^{\varepsilon}$ be as in Proposition~\ref{proposition: perturbed fixed point existence}, and define
$$
P^{\varepsilon} := (1-\varepsilon)Q^{\varepsilon} + \varepsilon R^{\varepsilon}.
$$
The fixed point identity gives
$\mu^{P^{\varepsilon}} = P^{\varepsilon}\circ X^{-1} = \nu^{\varepsilon}$.
Moreover,
$$
Q^{\varepsilon} \in \mathcal{R}^{*}(\nu^{\varepsilon}) \subset
\mathcal{R}(\nu^{\varepsilon}),
\qquad
R^{\varepsilon} \in \mathcal{R}_{\eta}(\nu^{\varepsilon}) \subset
\mathcal{R}(\nu^{\varepsilon}).
$$
By convexity of $\mathcal{R}(\nu^{\varepsilon})$, we have
$P^{\varepsilon} \in \mathcal{R}(\mu^{P^{\varepsilon}})$ and $Q^{\varepsilon} \in \mathcal{R}^{*}(\mu^{P^{\varepsilon}})$.
For every nonempty open set $O\subset\mathcal{V}[A]$,
$$
\begin{aligned}
\bigl(P^{\varepsilon}\circ\Lambda^{-1}\bigr)(O) =
(1-\varepsilon) \bigl(Q^{\varepsilon}\circ\Lambda^{-1}\bigr)(O) +
\varepsilon \bigl(R^{\varepsilon}\circ\Lambda^{-1}\bigr)(O) \geq
\varepsilon\eta(O)>0.
\end{aligned}
$$
Hence $P^{\varepsilon}\circ\Lambda^{-1}$ has full support on $\mathcal{V}[A]$.

Choose $\varepsilon_j\downarrow0$. Relative compactness of $\operatorname{Ran}(\mathcal{R})$ yields, along a subsequence,
$Q^{\varepsilon_j} \rightarrow \widehat{P}$ in $\mathcal{P}^{p}(\Omega[A])$. 
The standard mixture coupling gives
$$
W_{p,\Omega[A]}^{p} \left(P^{\varepsilon_j},Q^{\varepsilon_j}\right) \leq
\varepsilon_j W_{p,\Omega[A]}^{p} \left(R^{\varepsilon_j},Q^{\varepsilon_j}\right).
$$
Since $R^{\varepsilon_j}$ and $Q^{\varepsilon_j}$ belong to
$\operatorname{Ran}(\mathcal R)$, and the closure of
$\operatorname{Ran}(\mathcal R)$ is compact in
$\mathcal P^p(\Omega[A])$, the distances on the right-hand side are
uniformly bounded. Hence $W_{p,\Omega[A]}\left(P^{\varepsilon_j},Q^{\varepsilon_j}\right)\rightarrow 0.$
Since $Q^{\varepsilon_j}\to\widehat P$ in $\mathcal P^p(\Omega[A])$, it follows that
$P^{\varepsilon_j}\rightarrow\widehat P$ in $\mathcal P^p(\Omega[A])$.

For every $j$, we have $Q^{\varepsilon_j} \in \mathcal{R}^{*}(\mu^{P^{\varepsilon_j}})$.
Continuity of the state projection gives $\mu^{P^{\varepsilon_j}} \rightarrow \mu^{\widehat{P}}$ in $\mathcal{P}^{p}(\mathcal{C}^{d})$. The closed graph property of $\mathcal{R}^{*}$ then yields $\widehat{P} \in \mathcal{R}^{*}(\mu^{\widehat{P}})$.
Thus $\widehat{P}$ is a relaxed MFG equilibrium. Finally, the sequences
$\bigl(P^{\varepsilon_j}\bigr)_{j\geq1}$ and $\bigl(Q^{\varepsilon_j}\bigr)_{j\geq1}$ satisfy Definition~\ref{definition: thp}, and therefore $\widehat{P}$ is a THP equilibrium.  
\end{proof}

\section{Existence under the general assumptions}
\label{section: general existence}

In this section we prove Theorem~\ref{thm: general thp} under Assumption~\ref{ass_standing}. We combine the perturbed fixed point construction of Section~\ref{section: compact-bounded existence} with the truncation method of \cite[Section~5]{lacker2015mean}. The proof has two steps. For each fixed perturbation level $\varepsilon_j>0$, we first remove the truncation by letting $n\to\infty$. We then let $\varepsilon_j\to0$ and obtain the two sequences required by Definition~\ref{definition: thp}.

\subsection{Truncated perturbed problems}
\label{subsection: truncated perturbed problems}

Let $c_1$ be the constant in Assumption~\ref{ass_standing}\textnormal{(A2)}. For $n\geq1$, set
$$
r_n:=\left(\frac{n}{2c_1}\right)^{1/2},
\qquad
A_n:=A\cap\bar{B}(0,r_n).
$$
Since $A$ is nonempty and closed, there exists $n_0$ such that $A_n$ is nonempty and compact for every $n\geq n_0$. Let $b_n$ and $\sigma_n$ be the pointwise Euclidean projections of $b$ and $\sigma$ onto the closed balls of radius $n$ in $\mathbb{R}^d$ and $\mathbb{R}^{d\times m}$, respectively. As in \cite[Section~5]{lacker2015mean}, the truncated data $(b_n,\sigma_n,f,g,A_n,\lambda)$ satisfy Assumptions~\ref{ass_standing} and~\ref{ass_compact-bounded}, with the relevant bounds chosen uniformly in $n$. We identify $\mathcal{V}[A_n]$ and $\Omega[A_n]$ with their natural images in $\mathcal{V}[A]$ and $\Omega[A]$, and denote by $\mathcal{R}_n(\mu)$ and $\mathcal{R}_n^*(\mu)$ the admissible and optimal law correspondences associated with the truncated data. The reward functional is not truncated.

Let $\eta\in\mathcal{P}^{p}(\mathcal{V}[A])$
be the probability measure constructed in Proposition~\ref{proposition: global full-support law appendix}. Thus $\eta$ has full support and
$$
\int_{\mathcal{V}[A]} \left(\int_0^T\int_A |a|^{p'}q_t(da)\,dt\right)\eta(dq)<\infty.
$$
Fix $a_0\in A$ and increase $n_0$, if necessary, so that
$a_0\in A_n$ for every $n\ge n_0$. In particular,
$|a_0|\le r_n$ for every $n\ge n_0$. If $a\notin A_n$, then
$|a|>r_n\ge |a_0|$; consequently, the map defined below satisfies
$|\pi_n(a)|\le |a|$ for every $a\in A$. 

Define
$$
\pi_n(a) :=
\begin{cases}
a, & a\in A_n,\\
a_0, & a\notin A_n,
\end{cases}
$$
and, for $q\in\mathcal{V}[A]$,
$$
(\Pi_nq)(dt,da) := dt\,q_t\circ\pi_n^{-1}(da).
$$
Set $\eta_n:=\eta\circ\Pi_n^{-1}$.
By Proposition~\ref{proposition: eta-n approximation appendix}, when $\eta_n$ is viewed as a probability measure on $\mathcal{V}[A]$, $\eta_n\rightarrow\eta$ in $\mathcal{P}^{p}(\mathcal{V}[A])$, and
$$
\sup_{n\geq n_0} \int_{\mathcal{V}[A_n]}\left(\int_0^T\int_{A_n}|a|^{p'}q_t(da)\,dt
\right)\eta_n(dq)<\infty.
$$
Fix a sequence $(\varepsilon_j)_{j\geq1}\subset(0,1)$ such that $\varepsilon_j\downarrow0$. For $\nu\in\mathcal{P}^p(\mathcal{C}^d)$, define
$$
\mathcal{R}_{n,\eta_n}(\nu) := 
\left\{R\in\mathcal{R}_n(\nu):R\circ\Lambda^{-1}=\eta_n\right\}.
$$
Since $\eta$ is countably supported, so is $\eta_n$. Moreover, $\eta_n\in\mathcal{P}^{p}(\mathcal{V}[A_n])$. The arguments of
Propositions~\ref{proposition: Reta properties}
and~\ref{proposition: perturbed fixed point existence} therefore apply to the truncated data with $\eta$ replaced by $\eta_n$ and $\mathcal{R}_{\eta}$ replaced by $\mathcal{R}_{n,\eta_n}$.
Full support of $\eta_n$ is not required for the fixed point argument.

\begin{proposition}[Perturbed fixed points for the truncated problems]
\label{proposition: truncated perturbed fixed points}
For every $j\geq1$ and $n\geq n_0$, there exist laws $Q^{j,n}\in\mathcal{R}_n^*(\nu^{j,n})$ and $R^{j,n}\in\mathcal{R}_{n,\eta_n}(\nu^{j,n})$,
for some $\nu^{j,n}\in\mathcal{P}^p(\mathcal{C}^d)$, such that, with $P^{j,n} := (1-\varepsilon_j)Q^{j,n} + \varepsilon_jR^{j,n}$, we have
$$
\mu^{P^{j,n}}=\nu^{j,n},
\qquad
P^{j,n}\in\mathcal{R}_n(\mu^{P^{j,n}}),
\qquad
Q^{j,n}\in\mathcal{R}_n^*(\mu^{P^{j,n}}),
$$
and
$$
P^{j,n}\circ\Lambda^{-1} = (1-\varepsilon_j)Q^{j,n}\circ\Lambda^{-1} + \varepsilon_j\eta_n.
$$
\end{proposition}

\begin{proof}
Apply the perturbed fixed point argument of Section~\ref{section: compact-bounded existence} to the truncated data and the prescribed marginal $\eta_n$. The fixed point identity gives $\mu^{P^{j,n}}=\nu^{j,n}$. The remaining claims follow from the convexity of $\mathcal{R}_n(\nu^{j,n})$ and the linearity of the relaxed control projection.  
\end{proof}

\begin{proposition}[Uniform estimates and compactness]
\label{proposition: uniform estimates and compactness}
There exists a constant $C>0$, independent of $j$ and $n$, such that
$$
\sup_{j\geq1,\ n\geq n_0} \mathbb{E}^{U^{j,n}} \left[ \|X\|_T^{p'} +
\int_0^T|\Lambda_t|^{p'}\,dt \right] \leq C
$$
for each $U^{j,n}\in \left\{P^{j,n},Q^{j,n},R^{j,n}\right\}$. 

Moreover, the family $\mathcal{K} := \left\{P^{j,n},Q^{j,n},R^{j,n}:j\geq1,\ n\geq n_0
\right\}$ is relatively compact in $\mathcal{P}^p(\Omega[A])$.
\end{proposition}

\begin{proof}
The prescribed marginal of $R^{j,n}$ and the uniform $p'$-moment bound for $\eta_n$ control the perturbing laws. Comparing the optimal law $Q^{j,n}$ with a constant-control deviation, the coercive bound in Assumption~\ref{ass_standing}\textnormal{(A3)} yields a uniform bound on the control moments of order $p'$ of $Q^{j,n}$. The state estimate of Lemma~\ref{lemma: state estimate}, the self-consistency of $P^{j,n}$, and Young's inequality then close the estimates for all three families. The details are given in Proposition~\ref{prop: appendix uniform moment estimates}; the argument adapts the proof of \cite[Lemma~5.1]{lacker2015mean} to the coupled laws $P^{j,n},Q^{j,n},R^{j,n}$. 

The uniform $p'$-moment bounds imply tightness of the relaxed control marginals. The uniform increment estimate for the state processes verifies Aldous' condition and therefore yields tightness of the state
marginals in $\mathcal{P}(\mathcal{C}^d)$ by
\cite[Theorem~16.10]{billingsley2009convergence}.
Since $p'>p$, the same moment bounds provide the uniform
$p$-integrability required by
\cite[Propositions~B.3-B.4]{lacker2015mean}. Consequently, $\mathcal{K}$ is relatively compact in $\mathcal{P}^p(\Omega[A])$.
See Proposition~\ref{proposition: appendix relative compactness}
for the details.  
\end{proof}

\subsection{Removing the truncation}
\label{subsection: fixed perturbation limit}

Fix $j\geq1$. By Proposition~\ref{proposition: uniform estimates and compactness}, there exist a sequence $(n_k)_k$, possibly depending on $j$, with $n_k\to\infty$, and laws $Q^j,R^j\in\mathcal{P}^p(\Omega[A])$ such that
$$
Q^{j,n_k}\rightarrow Q^j,
\qquad
R^{j,n_k}\rightarrow R^j
$$
in $\mathcal{P}^p(\Omega[A])$. Define
$$
P^j := (1-\varepsilon_j)Q^j + \varepsilon_jR^j,
\qquad
\nu^j:=\mu^{P^j}.
$$
Then
$$
P^{j,n_k}\rightarrow P^j,
\qquad
\nu^{j,n_k}=\mu^{P^{j,n_k}}\rightarrow\nu^j
$$
in the corresponding $p$-Wasserstein spaces.

\begin{proposition}[Limit at a fixed perturbation level]
\label{proposition: fixed perturbation limit}
For every $j\geq1$,
$$
P^j\in\mathcal{R}(\mu^{P^j}),
\qquad
Q^j\in\mathcal{R}^*(\mu^{P^j}),
$$
and
$$
R^j\circ\Lambda^{-1}=\eta.
$$
Consequently, $\operatorname{supp} \left(P^j\circ\Lambda^{-1}\right) = \mathcal{V}[A].$
\end{proposition}

\begin{proof}
Proposition~\ref{proposition: stability truncated app}, applied first to $Q^{j,n_k}$ and then to $R^{j,n_k}$, gives $Q^j,R^j\in\mathcal{R}(\nu^j)$.
Moreover,
$$
R^{j,n_k}\circ\Lambda^{-1}=\eta_{n_k}\rightarrow\eta,
$$
and continuity of the relaxed control projection yields $R^j\circ\Lambda^{-1}=\eta.$
By convexity of $\mathcal{R}(\nu^j)$,
$$
P^j = (1-\varepsilon_j)Q^j + \varepsilon_jR^j \in
\mathcal{R}(\nu^j) = \mathcal{R}(\mu^{P^j}).
$$
It remains to prove that $Q^j\in\mathcal R^*(\nu^j)$.
Let $S\in\mathcal{R}(\nu^j)$. There is nothing to prove when $J(\nu^j,S)=-\infty$. Otherwise, Proposition~\ref{proposition: truncated deviation approximation app} gives a sequence of admissible deviations $S^k\in\mathcal{R}_{n_k}(\nu^{j,n_k})$
such that $S^k\rightarrow S$ in $\mathcal{P}^p(\Omega[A])$ and $J(\nu^{j,n_k},S^k) \rightarrow J(\nu^j,S).$
Since $Q^{j,n_k}\in\mathcal{R}_{n_k}^*(\nu^{j,n_k})$,
$$
J(\nu^{j,n_k},Q^{j,n_k}) \geq J(\nu^{j,n_k},S^k).
$$
Proposition~\ref{proposition: J upper semicontinuity} gives
$$
\begin{aligned}
J(\nu^j,Q^j) \geq \limsup_{k\to\infty} J(\nu^{j,n_k},Q^{j,n_k}) \geq J(\nu^j,S).
\end{aligned}
$$
Thus $Q^j\in\mathcal{R}^*(\nu^j)$.
Finally, $P^j\circ\Lambda^{-1} = (1-\varepsilon_j)Q^j\circ\Lambda^{-1} + \varepsilon_j\eta$.
Hence, for every nonempty open set $O\subset\mathcal{V}[A]$, we have
$\bigl(P^j\circ\Lambda^{-1}\bigr)(O) \geq \varepsilon_j\eta(O)>0$,
which proves the full support property.  
\end{proof}

\subsection{Vanishing perturbations}
\label{subsection: vanishing perturbation general}

\begin{proof}[Proof of Theorem~\ref{thm: general thp}]
By lower semicontinuity, the estimates of Proposition~\ref{proposition: uniform estimates and compactness} pass to the limits obtained above. In particular,
$$
\sup_{j\geq1} \mathbb{E}^{S^j} \left[ \|X\|_T^{p'} +
\int_0^T|\Lambda_t|^{p'}\,dt\right] <\infty
$$
for $S^j\in\{P^j,Q^j,R^j\}.$
Moreover, $(Q^j)_{j\geq1}$ is relatively compact in
$\mathcal{P}^p(\Omega[A])$. Indeed, for each $j$, the law $Q^j$
is the $W_p$-limit of a sequence $(Q^{j,n_k})_k$ contained in the
relatively compact family of Proposition~\ref{proposition: uniform estimates and compactness}.
Hence $Q^j\in\overline{\mathcal K}$ for every $j\geq1,$
where $\overline{\mathcal K}$ is compact in $\mathcal{P}^p(\Omega[A])$.
Passing to a subsequence, not relabelled, there exists
$\widehat P\in\mathcal{P}^p(\Omega[A])$ such that $Q^j\rightarrow\widehat P$
in $\mathcal{P}^p(\Omega[A])$.
Since
$P^j = (1-\varepsilon_j)Q^j + \varepsilon_jR^j$,
the coupling which pairs the common $Q^j$ component identically gives
\begin{equation*}
W_{p,\Omega[A]}^p(P^j,Q^j)
\leq
\varepsilon_j
W_{p,\Omega[A]}^p(R^j,Q^j).
\end{equation*}
The uniform $p'$-moment estimates imply
$$
\sup_j W_{p,\Omega[A]}^p(R^j,Q^j) <\infty.
$$
Therefore, $P^j\rightarrow\widehat P$
in $\mathcal{P}^p(\Omega[A])$, and consequently $\mu^{P^j} \rightarrow \mu^{\widehat P}$ in $\mathcal{P}^p(\mathcal{C}^d)$.
Since $P^j\in\mathcal{R}(\mu^{P^j}),$
Proposition~\ref{proposition: stability varying environments app} yields
$\widehat P\in\mathcal{R}(\mu^{\widehat P})$.

Let $S\in\mathcal{R}(\mu^{\widehat P})$. Again, there is nothing to prove when $J(\mu^{\widehat P},S)=-\infty$. Otherwise, Proposition~\ref{proposition: deviation approximation varying environments app}, applied with
$\mu^j:=\mu^{P^j}$ and $\mu:=\mu^{\widehat P}$,
gives laws $S^j\in\mathcal{R}(\mu^{P^j})$ such that $S^j\rightarrow S$ in $\mathcal{P}^p(\Omega[A])$ and $J(\mu^{P^j},S^j) \rightarrow J(\mu^{\widehat P},S).$
For every $j$, $Q^j\in\mathcal{R}^*(\mu^{P^j})$,
and hence $J(\mu^{P^j},Q^j) \geq J(\mu^{P^j},S^j)$.
Proposition~\ref{proposition: J upper semicontinuity} gives
$$
\begin{aligned}
J(\mu^{\widehat P},\widehat P) \geq \limsup_{j\to\infty} J(\mu^{P^j},Q^j)\geq J(\mu^{\widehat P},S).
\end{aligned}
$$
Since $S$ was arbitrary, one has $\widehat P\in\mathcal{R}^*(\mu^{\widehat P})$.
Thus $\widehat P$ is a relaxed MFG equilibrium.

Finally,
$$
P^j\rightarrow\widehat P,
\qquad
Q^j\rightarrow\widehat P
$$
in $\mathcal{P}^p(\Omega[A])$, and Proposition~\ref{proposition: fixed perturbation limit} shows that, for every $j$,
$$
P^j\in\mathcal{R}(\mu^{P^j}),
\qquad
Q^j\in\mathcal{R}^*(\mu^{P^j}),
$$
while $P^j\circ\Lambda^{-1}$ has full support on $\mathcal{V}[A]$. Therefore the two sequences $(P^j)_{j\geq1}$ and $(Q^j)_{j\geq1}$ satisfy Definition~\ref{definition: thp}, and $\widehat P$ is a THP equilibrium.  
\end{proof}

\section{A selection example}
\label{section: selection example}
We conclude with a one-dimensional example showing that trembling-hand perfection can eliminate a relaxed MFG equilibrium. The state dynamics are linear, the interaction depends only on the terminal population mean, and the one-sided control set forces every admissible perturbation whose relaxed control marginal has full support to move this mean in the same direction.

Let $A=[0,1]$, $\lambda=\delta_0$, and let $\sigma>0$ be constant. 
The model coefficients are
$$
b(t,x,\nu,a):=a,
\qquad
\sigma(t,x,\nu,a):=\sigma,
\qquad
f(t,x,\nu,a):=0,
$$
and
$$
g(x,\nu) := x \cdot \int_{\mathbb{R}}y\,\nu(dy).
$$
Thus, for an external population law
$\mu\in\mathcal{P}^2(\mathcal{C}^1)$, the representative agent maximizes
$$
J(\mu,P) = \left(\int_{\mathbb{R}}y\,\mu_T(dy)\right)\mathbb{E}^P[X_T]
$$
over $P\in\mathcal{R}(\mu)$.

The model satisfies Assumptions~\ref{ass_standing} and~\ref{ass_compact-bounded}, for instance with $p=2$ and any $p'>2$.

For $q\in\mathcal{V}[A]$, define
$$
F(q) := \int_0^T\int_A a\,q_t(da)\,dt.
$$
The map $F:\mathcal{V}[A]\to[0,T]$ is continuous because convergence in $\mathcal{V}[A]$ implies weak convergence of the normalized measures $q/T$, and $(t,a)\mapsto a$ is bounded and continuous on $[0,T]\times A$.

Let $q^0(dt,da):=dt\,\delta_0(da)$, $q^1(dt,da):=dt\,\delta_1(da)$,
and define
$$
\widehat P^i := \operatorname{Law}(q^i,X^i),
\qquad
i\in\{0,1\},
$$
where $X_t^0=\sigma W_t$ and $X_t^1=t+\sigma W_t$.

\begin{proposition}[Selection]
\label{proposition: selection example}
The model has exactly two relaxed MFG equilibria, $\widehat P^0$ and $\widehat P^1$, and $\widehat P^1$ is its unique THP equilibrium.
\end{proposition}

\begin{proof} \emph{Best responses.}
Fix $\mu\in\mathcal{P}^2(\mathcal{C}^1)$ and $P\in\mathcal{R}(\mu)$. By Proposition~\ref{proposition: martingale measure representation}, $P$ can be realized so that
$$
X_t = \int_0^t\int_A a\,\Lambda_s(da)\,ds + M_t,
$$
where $M$ is a square-integrable martingale with $M_0=0$. Hence $\mathbb{E}^P[X_T] = \mathbb{E}^P[F(\Lambda)]$,
and therefore, writing $\theta := \int_{\mathbb{R}}y\,\mu_T(dy)$,
we have $J(\mu,P) = \theta\mathbb{E}^P[F(\Lambda)]$.
Since $0\leq F(q)\leq T$ for every $q\in\mathcal{V}[A]$, it follows that
$$
\mathcal{R}^*(\mu) =
\begin{cases}
\{\widehat P^1\}, & \theta>0,\\
\mathcal{R}(\mu), & \theta=0,\\
\{\widehat P^0\}, & \theta<0.
\end{cases}
$$
Indeed, $F(q)=T$ implies $\int_0^T\int_A(1-a)\,q_t(da)\,dt=0$,
and hence $q=q^1$, while $F(q)=0$ similarly implies $q=q^0$. In each case, uniqueness of the corresponding state law identifies the unique optimal joint law.

\emph{Equilibria.}
Let $P$ be a relaxed MFG equilibrium and set
$\theta := \int_{\mathbb{R}}y\,\mu_T^P(dy)$.
The previous identity gives
$\theta = \mathbb{E}^P[X_T] = \mathbb{E}^P[F(\Lambda)] \geq0$.
If $\theta>0$, then $P=\widehat P^1$. If $\theta=0$, the
nonnegativity of $F$ implies $F(\Lambda)=0$, $P$-a.s., and therefore
$P=\widehat P^0$.

Conversely, $\widehat P^1$ is optimal in the environment it generates, whose terminal mean is $T$, while every admissible law is optimal in the zero-mean environment generated by $\widehat P^0$. Thus these are exactly the two relaxed MFG equilibria.

\emph{Failure of perfection of $\widehat P^0$.}
Suppose, to the contrary, that $\widehat P^0$ is a THP equilibrium. Then there exist sequences $(P^j)_{j\geq1}$ and $(Q^j)_{j\geq1}$ converging to $\widehat P^0$ in $\mathcal{P}^2(\Omega[A])$ such that, for every $j$,
$$
P^j\in\mathcal{R}(\mu^{P^j}),
\qquad
Q^j\in\mathcal{R}^*(\mu^{P^j}),
$$
and $P^j\circ\Lambda^{-1}$ has full support on $\mathcal{V}[A]$.

Consider the set $O := \left\{q\in\mathcal V[A]: F(q)>\frac{T}{2} \right\}$. This set is open by continuity of $F$, and it is nonempty since
$q^1\in O$. Writing $\rho^j:=P^j\circ\Lambda^{-1},$ the full-support condition gives $\rho^j(O)>0$. Since $P^j\in\mathcal R(\mu^{P^j})$, the preceding mean identity yields
$$
\begin{aligned}
\int_{\mathbb{R}} y\,\mu_T^{P^j}(dy) =
\mathbb{E}^{P^j}[F(\Lambda)]\geq \frac{T}{2}\rho^j(O) >0.
\end{aligned}
$$
It follows that $Q^j=\widehat P^1$ for every $j$, contradicting $Q^j\rightarrow\widehat P^0$. Hence $\widehat P^0$ is not a THP equilibrium.

\emph{Perfection of $\widehat P^1$.}
Let $\eta$ be a probability measure with full support on $\mathcal{V}[A]$, whose existence was noted in Subsection~\ref{subsection: full support perturbations}. On a probability space carrying an independent standard Brownian motion $W$ and a $\mathcal{V}[A]$-valued random variable $\Lambda$ with law $\eta$, define
$$
X_t := \int_0^t\int_A a\,\Lambda_s(da)\,ds + \sigma W_t, 
\qquad 
R:=\operatorname{Law}(\Lambda,X).
$$
Since the coefficients do not depend on the population law, we have $R\in\mathcal{R}(\mu)$ for every $\mu\in\mathcal{P}^2(\mathcal{C}^1)$.
Let $\varepsilon_j\downarrow0$ and define
$$
P^j := (1-\varepsilon_j)\widehat P^1 + \varepsilon_jR,
\qquad
Q^j:=\widehat P^1.
$$
Both $R$ and $\widehat P^1$ belong to $\mathcal{R}(\mu^{P^j})$, and convexity gives $P^j\in\mathcal{R}(\mu^{P^j})$.
Moreover, $P^j\circ\Lambda^{-1} = (1-\varepsilon_j)\delta_{q^1} + \varepsilon_j\eta$,
which has full support on $\mathcal{V}[A]$. Finally,
$$
\begin{aligned}
\int_{\mathbb{R}}y\,\mu_T^{P^j}(dy) =
(1-\varepsilon_j)T + \varepsilon_j\mathbb{E}^R[F(\Lambda)]\geq 
(1-\varepsilon_j)T >0.
\end{aligned}
$$
Thus, $Q^j=\widehat P^1 \in \mathcal{R}^*(\mu^{P^j})$.
Since $R,\widehat P^1\in\mathcal{P}^2(\Omega[A])$, the standard coupling gives
$$
W_{2,\Omega[A]}^2(P^j,\widehat P^1) \leq
\varepsilon_j W_{2,\Omega[A]}^2(R,\widehat P^1) \rightarrow0.
$$
Since $Q^j=\widehat P^1$ for every $j$, the two sequences satisfy
Definition~\ref{definition: thp}. Hence $\widehat P^1$ is a THP
equilibrium, and the preceding exclusion of $\widehat P^0$ proves
uniqueness.  
\end{proof}

\begin{remark}[Role of the control set]
\label{remark: selection control set}
The exclusion of $\widehat P^0$ relies on the one-sided control
constraint. Indeed, suppose instead that $A=[-1,1]$. Let
$(q^k)_{k\geq1}$ be a countable dense subset of
$\mathcal{V}[-1,1]$, and let $\widetilde q^k$ denote the relaxed
control obtained from $q^k$ by replacing each action $a$ with $-a$.
Then
$$
\eta := \sum_{k=1}^{\infty} 2^{-k-1}\left(\delta_{q^k} + \delta_{\widetilde q^k}\right)
$$
has full support on $\mathcal{V}[-1,1]$ and satisfies $\int_{\mathcal{V}[-1,1]}F(q)\,\eta(dq)=0.$
Now, construct $R$ from $\eta$ as in the proof above, and, for
$\varepsilon_j\downarrow0$, set
$$
P^j := (1-\varepsilon_j)\widehat P^0 + \varepsilon_jR,
\qquad
Q^j:=\widehat P^0.
$$
As before, $P^j$ is admissible and
$P^j\circ\Lambda^{-1}$ has full support. Moreover,
$\int_{\mathbb{R}}y\,\mu_T^{P^j}(dy) = 0$,
so every admissible law is optimal in the environment generated by
$P^j$, and in particular $Q^j\in\mathcal{R}^*(\mu^{P^j})$.
The standard mixture estimate gives $P^j\rightarrow\widehat P^0$, while
$Q^j=\widehat P^0$ for every $j$. Thus $\widehat P^0$ is a THP
equilibrium when $A=[-1,1]$. The selection mechanism therefore depends on the interaction between the full
support condition and the geometry of the control set.
\end{remark}

\section{Conclusion}
\label{section: conclusion}

We introduced a trembling-hand perfection refinement for stochastic mean field games formulated through relaxed controlled martingale problems. The refinement combines full-support perturbations with admissibility, ensuring that trembles represent feasible population behavior rather than exogenous perturbations of the control marginal. We proved existence under the general continuity, growth, and coercivity assumptions considered here, without requiring compact controls or bounded coefficients. The one-dimensional example further shows that the refinement can genuinely select among multiple relaxed MFG equilibria and that the resulting selection depends on the interaction between full support and the geometry of the control set.

The analysis is carried out at the mean field level and without common noise. Natural directions for further work include extensions to common-noise models, where conditional population laws would enter the formulation, and a finite-player foundation based on completely mixed $N$-player perturbations. The latter would in particular require understanding the interaction between the limits $N\to\infty$ and the vanishing-perturbation limit.

\newpage
\begin{appendix}

\section{Full-support perturbations}
\label{appendix: full-support perturbations}

\subsection{Support of the barycentric relaxed control}
\label{appendix: barycentric support}
\begin{proof}[Proof of Proposition~\ref{proposition: barycentric full support}]
Let $U\subset[0,T]\times A$ be nonempty and open. Choose a nonempty open
rectangle $I\times O\subset U$, fix $a_0\in O$, and let $q^0(dt,da)=dt\,\delta_{a_0}(da).$ Then
$$
q^0(U)\ge q^0(I\times O)=|I|>0.
$$
By the Portmanteau theorem, $q\mapsto q(U)$ is lower semicontinuous for
weak convergence and therefore for the stronger topology of $\mathcal V[A]$.
Thus $\mathcal O_U:=\{q:q(U)>q^0(U)/2\}$ is a nonempty open subset of $\mathcal V[A]$. Full support of $\eta$ gives
$\eta(\mathcal O_U)>0$, and
$$
\bar \eta(U)=\int q(U)\eta(dq) \ge \frac{q^0(U)}2\eta(\mathcal O_U)>0.
$$
Hence $\operatorname{supp}(\bar \eta)=[0,T]\times A$.

Fix a nonempty open set $O\subset A$ and suppose that $E:=\{t:\bar \eta_t(O)=0\}$ has positive Lebesgue measure. Then
$$
0=\bar \eta(E\times O)=\int q(E\times O)\eta(dq),
$$
so $q(E\times O)=0$ for $\eta$-almost every $q$. Choose $a_1\in O$, let $q^1(dt,da)=dt\,\delta_{a_1}(da),$ and choose a bounded continuous function $\chi:A\to[0,1]$ with $\chi(a_1)=1$ and $\operatorname{supp}(\chi)\subset O$. The map
$$
G(q):=\int_{[0,T]\times A}\chi(a)q(dt,da)
$$
is continuous on $\mathcal V[A]$. Every $q$ satisfying
$q(E\times O)=0$ satisfies $G(q)\le T-|E|$,
whereas $G(q^1)=T$. Therefore $\{q:G(q)>T-|E|\}$ is a nonempty open set of zero $\eta$-mass, contradicting full support.
Hence $\bar \eta_t(O)>0$ for almost every $t$.

Since $A$ is separable, take a countable base $(O_k)_{k\ge1}$ and remove
the union of the corresponding null sets. Outside one null set,
$\bar \eta_t(O_k)>0$ for every $k$, which is equivalent to
$\operatorname{supp}(\bar \eta_t)=A$.  
\end{proof}

\subsection{Construction and truncation of full-support perturbation laws}
\label{appendix: perturbation laws}
This subsection constructs the probability measure with full support and the truncated perturbation laws used in Subsection~\ref{subsection: truncated perturbed problems}. Throughout, $A$ is a nonempty closed subset of a finite-dimensional Euclidean space and $p'>p\geq1$.

For $r\geq1$ and $q\in\mathcal{V}[A]$, we use the notation
$$
|q_t|^r := \int_A|a|^r q_t(da).
$$

\begin{proposition}[Global perturbation law with full support]
\label{proposition: global full-support law appendix}
There exists $\eta\in\mathcal{P}^{p}(\mathcal{V}[A])$ such that $\operatorname{supp}(\eta)=\mathcal{V}[A]$
and
$$
\int_{\mathcal{V}[A]} \int_0^T|q_t|^{p'}\,dt\,\eta(dq)<\infty.
$$
\end{proposition}

\begin{proof}
Fix $a_0\in A$. For $q\in\mathcal{V}[A]$ and each integer
$\ell\geq |a_0|\vee1$, define $q^\ell(dt,da) := dt\,q_t^\ell(da)$,
where
$$
q_t^\ell := 
q_t\big|_{A\cap\bar B(0,\ell)} + q_t\bigl(A\setminus\bar B(0,\ell)\bigr)\delta_{a_0}.
$$
Since $|a_0|\leq\ell$, we have
$\int_0^T |q_t^\ell|^{p'}\,dt \leq T\ell^{p'} <\infty$.
The coupling that leaves $a$ unchanged when $|a|\leq\ell$ and
replaces it with $a_0$ otherwise gives
$$
d_{\mathcal{V}[A]}^p(q,q^\ell) \leq
\frac1T \int_0^T \int_{A\setminus\bar B(0,\ell)} |a-a_0|^p q_t(da)\,dt.
$$
The integrand converges pointwise to zero and is bounded by $2^{p-1}\bigl(|a|^p+|a_0|^p\bigr),$ which is integrable because $q\in\mathcal{V}[A]$. Hence $q^\ell\rightarrow q$ in $\mathcal{V}[A]$ as $\ell \to \infty$. Thus the relaxed controls with finite control
moment of order $p'$ are dense in $\mathcal{V}[A]$.

Since $\mathcal{V}[A]$ is separable, there exists a countable dense
family $(q^k)_{k\geq1}$ such that $\int_0^T |q_t^k|^{p'}\,dt<\infty$
for every $k\geq1$. Define $c>0$ by
$$
c^{-1} := \sum_{k=1}^{\infty} \frac{2^{-k}} {1+\int_0^T |q_t^k|^{p'}\,dt},
$$
and set
$$
\eta := c \sum_{k=1}^{\infty} \frac{2^{-k}} {1+\int_0^T |q_t^k|^{p'}\,dt} \delta_{q^k}.
$$
Every coefficient in this sum is strictly positive. Therefore, if
$O\subset\mathcal{V}[A]$ is nonempty and open, density of
$(q^k)_{k\geq1}$ gives an index $k$ such that $q^k\in O$, and hence
$\eta(O)>0$. Thus $\operatorname{supp}(\eta)=\mathcal{V}[A]$.
Moreover,
$$
\begin{aligned}
\int_{\mathcal{V}[A]} \left(\int_0^T |q_t|^{p'}\,dt\right)\eta(dq)=
c \sum_{k=1}^{\infty} 2^{-k}\frac{\int_0^T |q_t^k|^{p'}\,dt}{1+\int_0^T |q_t^k|^{p'}\,dt} \leq
c \sum_{k=1}^{\infty}2^{-k}<\infty.
\end{aligned}
$$
It remains to verify that
$\eta\in\mathcal{P}^{p}(\mathcal{V}[A])$. Taking the constant relaxed
control $dt\,\delta_{a_0}(da)$ as reference, the natural coupling gives
\begin{equation}
\label{eq:estimate d_V}
d_{\mathcal{V}[A]}^p
\bigl(q,dt\,\delta_{a_0}(da)\bigr)
\leq
\frac{2^{p-1}}{T}
\left(
\int_0^T |q_t|^p\,dt
+
T|a_0|^p
\right).
\end{equation}
Since $p'>p$, we have
$
\int_0^T |q_t|^p\,dt
\leq
T+\int_0^T |q_t|^{p'}\,dt
$.
Integrating the estimate \eqref{eq:estimate d_V} with respect to $\eta$ proves the
required finite $p$-Wasserstein moment.  
\end{proof}

Let $\pi_n$, $\Pi_n$, and $\eta_n$ be as in
Subsection~\ref{subsection: truncated perturbed problems}.

\begin{proposition}[Approximation by truncated perturbation laws]
\label{proposition: eta-n approximation appendix}
For every $n\geq n_0$, the measure $\eta_n$ belongs to $\mathcal{P}^{p}(\mathcal{V}[A_n])$. Under the natural embedding into $\mathcal{V}[A]$, we have $\eta_n\rightarrow\eta$ in $\mathcal{P}^{p}(\mathcal{V}[A])$, and
$$
\sup_{n\geq n_0} \int_{\mathcal{V}[A_n]} \int_0^T|q_t|^{p'}\,dt\,\eta_n(dq) <\infty.
$$
\end{proposition}

\begin{proof}
The map $\Pi_n$ is Borel measurable because it is induced by the
Borel map $(t,a)\mapsto(t,\pi_n(a)).$
Since $A_n$ is compact, so is $\mathcal{V}[A_n]$, and therefore $\eta_n\in\mathcal{P}^{p}(\mathcal{V}[A_n]).$ Moreover, $|\pi_n(a)|\leq|a|$ for every $a\in A$ and $n\geq n_0$.
Hence
$$
\begin{aligned}
\int_{\mathcal{V}[A_n]}\left( \int_0^T |q_t|^{p'}\,dt\right) \eta_n(dq)
&=\int_{\mathcal{V}[A]} \int_0^T\int_A |\pi_n(a)|^{p'}q_t(da)\,dt\,\eta(dq)\\
&\leq\int_{\mathcal{V}[A]}\left(\int_0^T |q_t|^{p'}\,dt\right)\eta(dq).
\end{aligned}
$$
This proves the uniform control moment bound of order $p'$.

For convergence, use the coupling of $\eta$ and $\eta_n$ induced by
$q\mapsto(q,\Pi_nq)$. Then
$$
W_{p,\mathcal{V}[A]}^p(\eta_n,\eta) \leq\frac1T \int_{\mathcal{V}[A]} 
\int_0^T \int_A|\pi_n(a)-a|^p q_t(da)\,dt\,\eta(dq).
$$
For every $a\in A$, $\pi_n(a)\rightarrow a$,
and $|\pi_n(a)-a|^p \leq 2^p|a|^p$.
Moreover,
$$
\begin{aligned}
\int_{\mathcal{V}[A]} \int_0^T |q_t|^p\,dt\,\eta(dq)
&\leq T+ \int_{\mathcal{V}[A]}\left(\int_0^T |q_t|^{p'}\,dt\right)\eta(dq) <\infty.
\end{aligned}
$$
Dominated convergence therefore yields
$W_{p,\mathcal{V}[A]}(\eta_n,\eta) \rightarrow0$.
 
\end{proof}

\section{Uniform estimates and compactness for the truncated problems}
\label{appendix: uniform estimates compactness}

This appendix proves the uniform estimates and relative compactness asserted in Proposition~\ref{proposition: uniform estimates and compactness}. The laws $P^{j,n}$, $Q^{j,n}$, and $R^{j,n}$ are those constructed in Proposition~\ref{proposition: truncated perturbed fixed points}. Throughout, Assumption~\ref{ass_standing} is in force, and all constants are independent of $j$ and $n$.

\begin{proposition}[Uniform moment estimates]
\label{prop: appendix uniform moment estimates}
There exists a constant $C>0$, independent of $j$ and $n$, such that
$$
\sup_{j\geq1,\ n\geq n_0} \mathbb{E}^{U^{j,n}} \left[ \|X\|_T^{p'} +
\int_0^T|\Lambda_t|^{p'}\,dt\right]\leq C
$$
for each $U^{j,n}\in \left\{P^{j,n},Q^{j,n},R^{j,n} \right\}$.
\end{proposition}

\begin{proof}
All constants below are independent of $j$ and $n$. Metric projection onto a closed Euclidean ball is $1$-Lipschitz and does not increase the norm. Hence,
the truncated coefficients satisfy the Lipschitz and growth estimates in Assumption~\ref{ass_standing} with constants independent of $j$ and $n$.

Set
$$
K^{j,n} := \mathbb E^{Q^{j,n}} \left[\int_0^T |\Lambda_t|^{p'}dt \right],
\qquad
Y^{j,n} := \mathbb E^{P^{j,n}}[\|X\|_T^p].
$$
The prescribed marginal of $R^{j,n}$ and
Proposition~\ref{proposition: eta-n approximation appendix} imply
\begin{equation}
\label{eq: R control bound revised}
\sup_{j,n}
\mathbb E^{R^{j,n}} \left [
\int_0^T |\Lambda_t|^{p'}dt \right]
<\infty.
\end{equation}
Fix $a_0\in A_{n_0}$ and let $\widetilde Q^{j,n}\in\mathcal R_n(\nu^{j,n})$ be generated by the constant control $a_0$. Since $\nu^{j,n}=\mu^{P^{j,n}}$, Lemma~\ref{lemma: state estimate},  applied to the truncated data at order $p$, gives
$\mathbb E^{\widetilde Q^{j,n}}[\|X\|_T^p] \le C\bigl(1+Y^{j,n}\bigr)$.
Using the lower reward bounds in Assumption~\ref{ass_standing}(A3)
and the fact that $a_0$ is fixed, we obtain
\begin{equation}
\label{eq: comparison lower revised}
J(\nu^{j,n},\widetilde Q^{j,n})
\ge
-C\bigl(1+Y^{j,n}\bigr).
\end{equation}
The optimality of $Q^{j,n}$, the coercive upper bound for $f$, and the growth
bound for $g$ yield, after comparison with
\eqref{eq: comparison lower revised},
\begin{equation}
\label{eq: coercive control revised}
K^{j,n}
\le
C\left(
1+Y^{j,n}
+\mathbb E^{Q^{j,n}}[\|X\|_T^p]
\right).
\end{equation}
Applying Lemma~\ref{lemma: state estimate} again at order $p$ to $Q^{j,n}\in\mathcal R_n(\nu^{j,n})$ gives
$$
\mathbb E^{Q^{j,n}}[\|X\|_T^p] \le C\left(1+Y^{j,n} +
\mathbb E^{Q^{j,n}} \left [\int_0^T |\Lambda_t|^pdt \right]\right).
$$
For every $\delta>0$, there exists $C_\delta>0$ such that $r^p\le C_\delta+\delta r^{p'},$ for $r\ge0.$
Hence
$$
\mathbb E^{Q^{j,n}}[\|X\|_T^p] \le C_\delta\bigl(1+Y^{j,n}\bigr) + C\delta K^{j,n}.
$$
Substituting this estimate into
\eqref{eq: coercive control revised} and choosing $\delta>0$ sufficiently
small to absorb the last term yields
\begin{equation}
\label{eq: U by Y revised}
K^{j,n}
\le
C\bigl(1+Y^{j,n}\bigr).
\end{equation}
We next close the state estimate at order $p$. Since
$P^{j,n}\in\mathcal R_n(\mu^{P^{j,n}})$, the self-consistent form of the
state estimate gives
$$
Y^{j,n} \le C\left(1+\mathbb E^{P^{j,n}} \left [\int_0^T |\Lambda_t|^pdt \right]\right).
$$
Using
$P^{j,n} = (1-\varepsilon_j)Q^{j,n} + \varepsilon_jR^{j,n}$,
Young's inequality, \eqref{eq: R control bound revised}, and
\eqref{eq: U by Y revised}, we obtain, for every $\delta>0$,
$$
\mathbb E^{P^{j,n}} \left [\int_0^T |\Lambda_t|^pdt \right] \le
C_\delta+C\delta Y^{j,n}.
$$
Consequently,
$Y^{j,n} \le C_\delta+C\delta Y^{j,n}$.
Choosing $\delta$ sufficiently small yields
\begin{equation}
\label{eq: uniform p state revised}
\sup_{j,n}Y^{j,n}<\infty.
\end{equation}
Combining \eqref{eq: U by Y revised} and
\eqref{eq: uniform p state revised} gives
$\sup_{j,n} \mathbb E^{Q^{j,n}}\left[\int_0^T |\Lambda_t|^{p'}dt\right] <\infty$.
The mixture identity and \eqref{eq: R control bound revised} then imply
$\sup_{j,n} \mathbb E^{P^{j,n}} \left[\int_0^T |\Lambda_t|^{p'}dt \right] <\infty $.
Thus the control moments of order $p'$ are uniformly bounded for all three
families.

Since $P^{j,n}\in\mathcal R_n(\mu^{P^{j,n}})$, the self-consistent estimate in Lemma~\ref{lemma: state estimate}, applied at order $p'$, gives
$\sup_{j,n} \mathbb E^{P^{j,n}}[\|X\|_T^{p'}] <\infty$.
Since $\nu^{j,n}=\mu^{P^{j,n}}$, it follows that
$$
\sup_{j,n} \|\nu^{j,n}\|_{T,p'}^{p'} =
\sup_{j,n} \mathbb E^{P^{j,n}}[\|X\|_T^{p'}] <\infty.
$$
Finally, applying the state estimate at order $p'$ to $Q^{j,n}$ and
$R^{j,n}$ in the common environment $\nu^{j,n}$, together with the uniform
$p'$-moment bounds for their controls, yields
$$
\sup_{j,n} \mathbb E^{U^{j,n}}[\|X\|_T^{p'}] <\infty,
\qquad
U^{j,n}\in\{Q^{j,n},R^{j,n}\}.
$$
This completes the proof.  
\end{proof}

\begin{proposition}[Relative compactness of the truncated laws]
\label{proposition: appendix relative compactness}
The family
$$
\mathcal{K} := \left\{P^{j,n},Q^{j,n},R^{j,n}:j\geq1,\ n\geq n_0\right\}
$$
is relatively compact in $\mathcal{P}^{p}(\Omega[A])$.
\end{proposition}

\begin{proof}
Let $U^{j,n}$ denote any of the three families
$P^{j,n}$, $Q^{j,n}$, and $R^{j,n}$. Proposition~\ref{prop: appendix uniform moment estimates} gives uniform control and state moments of order $p'$.

We first prove tightness of the state marginals. Since
$\nu^{j,n}=\mu^{P^{j,n}}$, the uniform state moment bound yields
$$ 
\sup_{j,n}\sup_{t\le T} m_p^p(\nu_t^{j,n}) \le 
\sup_{j,n} \mathbb E^{P^{j,n}}[\|X\|_T^p] <\infty. 
$$
For each $j,n$, regard the external environment as part of the coefficients
by setting
$$
b^{j,n}(t,x,a) := b_n(t,x,\nu_t^{j,n},a),
$$
and defining $\sigma^{j,n}$ analogously. By
Assumption~\ref{ass_standing}(A2), the fact that the metric projections
defining $b_n$ and $\sigma_n$ do not increase the norm, and the uniform
bound on the moments of $\nu^{j,n}$ above, these coefficients satisfy the
growth bounds required in \cite[Proposition~B.4]{lacker2015mean}, with
constants independent of $j$ and $n$.

Accordingly, the argument of \cite[Proposition~B.4]{lacker2015mean} applies uniformly to the present
family. In particular, the conditional Burkholder--Davis--Gundy inequality and H\"older's inequality yield
$$
\lim_{\delta\downarrow0} \sup_{j,n} \sup_{\tau}
\mathbb E^{U^{j,n}}\left[ \left|X_{(\tau+\delta)\wedge T}-X_\tau\right|^p \right] =0,
$$
where the supremum is over all stopping times $\tau$ taking values in $[0,T]$.
Together with the uniform $p'$-moment bound for $\|X\|_T$, which gives
compact containment, Aldous' criterion yields tightness of $\left\{U^{j,n}\circ X^{-1}: j\ge1,\ n\ge n_0 \right\}$ in $\mathcal P(\mathcal C^d)$.

Having established tightness of the state marginals, the uniform
$p'$-moment bounds for the state and control coordinates, together with
$p'>p$, allow us to apply
\cite[Proposition~B.3]{lacker2015mean}. It follows that $\mathcal K := \left\{P^{j,n},Q^{j,n},R^{j,n}: j\ge1,\ n\ge n_0\right\}$ is relatively compact in $\mathcal P^p(\Omega[A])$.  
\end{proof}

\section{Stability under truncation and in varying environments}
\label{appendix: stability}

This appendix proves the two stability results used in Proposition~\ref{proposition: fixed perturbation limit} and in the proof of Theorem~\ref{thm: general thp}. The first removes the truncation at a fixed perturbation level. The second treats convergence after the original coefficients and control space have been restored.

We retain the notation $\mathcal{L}^{\mu,a}_t$ and $M^{\mu,\varphi}_t$
introduced in Subsection~\ref{subsection: martingale problem}.
For $n\geq n_0$ and $\varphi\in C_c^\infty(\mathbb{R}^d)$, define the truncated generator by
$$
\mathcal{L}_t^{n,\mu,a}\varphi(x) :=
b_n(t,x,\mu_t,a)^\top D\varphi(x) + \frac12 \operatorname{Tr}\left[\sigma_n\sigma_n^\top(t,x,\mu_t,a)D^2\varphi(x)\right].
$$
For $(q,x)\in\Omega[A]$, set
$$
M_t^{n,\mu,\varphi}(q,x) := \varphi(x_t) - \int_0^t\int_A \mathcal{L}_s^{n,\mu,a}\varphi(x_s) q_s(da)\,ds.
$$

\begin{proposition}[Stability of truncated admissible laws]
\label{proposition: stability truncated app}
Let $n_k\to\infty$, and assume $\mu^k\rightarrow\mu$
in $\mathcal{P}^{p}(\mathcal{C}^d),$ $S^k\rightarrow S$ in $\mathcal{P}^{p}(\Omega[A])$.
Suppose that $S^k\in\mathcal{R}_{n_k}(\mu^k)$
for every $k$, and
$$
\sup_k \left(\|\mu^k\|_{T,p'}^{p'} + \mathbb{E}^{S^k} \left[\|X\|_T^{p'} + \int_0^T|\Lambda_t|^{p'}\,dt\right]\right)<\infty.
$$
Then $S\in\mathcal{R}(\mu)$.
\end{proposition}

\begin{proof}
The initial law passes to the limit because $(q,x)\mapsto x_0$ is
continuous. Moreover, the map $q\mapsto \int_0^T\int_A |a|^p q_t(da)\,dt$ is lower semicontinuous, so the control moment condition also passes to
the limit.

Fix $0\le r<t\le T$, $\varphi\in C_c^\infty(\mathbb R^d)$, and a bounded
continuous $\mathcal F_r$-measurable function $H$. Set
$$
\Delta_k:=\mathbb E^{S^k}\left[H\bigl(M_t^{\mu^k,\varphi}-M_r^{\mu^k,\varphi}\bigr)
\right].
$$
Since $S^k\in\mathcal R_{n_k}(\mu^k)$,
$$
\mathbb E^{S^k}\left[H\bigl(M_t^{n_k,\mu^k,\varphi} -M_r^{n_k,\mu^k,\varphi}\bigr)\right]=0.
$$
Hence
$$
|\Delta_k|\le \|H\|_\infty \mathbb E^{S^k} \left [ \int_r^t\int_A \left|\mathcal L_s^{n_k,\mu^k,a}\varphi(X_s) - \mathcal L_s^{\mu^k,a}\varphi(X_s) \right| \Lambda_s(da)\,ds \right ].
$$
The truncated coefficients agree with the original ones whenever the
latter remain inside the truncation balls. Assumption~\ref{ass_standing}(A2), the uniform $p'$-moment bound, and $p'>p\ge 1\vee p_\sigma$ imply, exactly as in
\cite[(5.8)-(5.10)]{lacker2015mean}, that the right-hand side converges
to zero. In particular, the diffusion contribution involves
$\sigma\sigma^\top$, whose growth order is $p_\sigma$.

It remains to pass to the limit in the untruncated martingale
functional. By Lemma~\ref{lemma: time marginal continuity}, we have $\sup_{s\in[0,T]}W_p(\mu_s^k,\mu_s)\rightarrow0$.
The continuity argument used in
\cite[Lemma~5.2]{lacker2015mean} for the untruncated generator therefore
gives
$$
\Delta_k\rightarrow \mathbb E^S\left[H\bigl(M_t^{\mu,\varphi}-M_r^{\mu,\varphi}\bigr)\right].
$$
Since $\Delta_k\to0$, the latter expectation is zero. A monotone-class
argument extends the identity to every bounded
$\mathcal F_r$-measurable test function. Hence
$M^{\mu,\varphi}$ is an $S$-martingale for every $\varphi$, and thus
$S\in\mathcal R(\mu)$.  
\end{proof}

\begin{proposition}[Stability under varying environments]
\label{proposition: stability varying environments app}
Let $\mu^j\rightarrow\mu$ in $\mathcal{P}^{p}(\mathcal{C}^d),$ $S^j\rightarrow S$ in $\mathcal{P}^{p}(\Omega[A])$.
Suppose that $S^j\in\mathcal{R}(\mu^j)$
for every $j$, and
$$
\sup_j\mathbb{E}^{S^j}\left[\|X\|_T^{p'}+\int_0^T|\Lambda_t|^{p'}\,dt\right]<\infty.
$$
Then $S\in\mathcal{R}(\mu)$.
\end{proposition}

\begin{proof}
The initial law and moment conditions pass to the limit exactly as in Proposition~\ref{proposition: stability truncated app}. For the martingale
condition, fix $0\le r<t\le T$, $\varphi\in C_c^\infty(\mathbb R^d)$ and a
bounded continuous $\mathcal F_r$-measurable function $H$. Since
$S^j\in\mathcal R(\mu^j)$,
$$
0=\mathbb E^{S^j}\left[H\bigl(M_t^{\mu^j,\varphi}-M_r^{\mu^j,\varphi}\bigr)\right].
$$
By Lemma~\ref{lemma: time marginal continuity}, $\sup_{s\in[0,T]}W_p(\mu_s^j,\mu_s)\rightarrow 0$.
The same continuity and uniform-integrability argument used in
Proposition~\ref{proposition: stability truncated app} for the untruncated
martingale functional therefore yields
$$
\mathbb E^{S^j}\left[H\bigl(M_t^{\mu^j,\varphi}-M_r^{\mu^j,\varphi}\bigr)\right]
\rightarrow\mathbb E^{S}\left[H\bigl(M_t^{\mu,\varphi}-M_r^{\mu,\varphi}\bigr)\right].
$$
Hence
$$
\mathbb E^{S}\left[H\bigl(M_t^{\mu,\varphi}-M_r^{\mu,\varphi}\bigr)\right]=0.
$$
A monotone-class argument extends the identity to every bounded
$\mathcal F_r$-measurable test function. Hence $M^{\mu,\varphi}$ is an
$S$-martingale for every $\varphi$, and therefore $S\in\mathcal R(\mu)$.  
\end{proof}

\section{Approximation of admissible deviations}
\label{appendix: deviation approximation}
This appendix proves the approximation results for admissible deviations used in Proposition~\ref{proposition: fixed perturbation limit} and in the proof of Theorem~\ref{thm: general thp}. The first approximates admissible deviations for the original problem by admissible deviations for the truncated problems. The second approximates admissible deviations in varying environments after the original coefficients and control space have been restored.

\begin{proposition}[Approximation of admissible deviations for the truncated problems]
\label{proposition: truncated deviation approximation app}
Fix $j\geq1$, and let $n_k\to\infty$ be such that $\nu^{j,n_k}\rightarrow\nu^j$
in $\mathcal{P}^{p}(\mathcal{C}^d)$ and $\sup_k\|\nu^{j,n_k}\|_{T,p'}^{p'}<\infty.$ 

If $S\in\mathcal{R}(\nu^j)$ satisfies $J(\nu^j,S)>-\infty,$ then there exist $S^k\in\mathcal{R}_{n_k}(\nu^{j,n_k})$
such that $S^k\rightarrow S$
in $\mathcal{P}^{p}(\Omega[A])$ and
$J(\nu^{j,n_k},S^k) \rightarrow J(\nu^j,S)$.
\end{proposition}

\begin{proof}
By lower semicontinuity, the assumption
$\sup_k \|\nu^{j,n_k}\|_{T,p'}<\infty$ implies that
$\|\nu^j\|_{T,p'}<\infty$. Since
$S\in\mathcal R(\nu^j)$ and $J(\nu^j,S)>-\infty$,
Lemma~\ref{lemma: finite value higher moments} yields
$$
\mathbb E^S\left[\|X\|_T^{p'}+\int_0^T|\Lambda_t|^{p'}dt\right]<\infty.
$$
Use Proposition~\ref{proposition: martingale measure representation} to realize $S$ on a filtered probability space $(\Omega',\mathcal F',(\mathcal F'_t)_{t\in[0,T]},P')$ carrying a predictable relaxed control $\Lambda$, a continuous adapted process $X$, and orthogonal martingale measures $N$ with intensity
$\Lambda_t(da)dt$, such that $S=P'\circ(\Lambda,X)^{-1}$.
In particular,
$$
\mathbb E^{P'}\left[\|X\|_T^{p'}+\int_0^T|\Lambda_t|^{p'}dt\right]<\infty.
$$

Define $\Lambda^k:=\Pi_{n_k}\Lambda$, so that $\Lambda_t^k=\Lambda_t\circ\pi_{n_k}^{-1}.$
For each $i=1,\ldots,m$, define the push-forward martingale measure
$N^{k,i}$ on $A_{n_k}\times[0,T]$ by
$$
N^{k,i}(B\times(s,t]):=N^i(\pi_{n_k}^{-1}(B)\times(s,t]),\qquad B\in\mathcal B(A_{n_k}).
$$
Then $N^k=(N^{k,1},\ldots,N^{k,m})$ is an orthogonal martingale
measure with intensity $\Lambda_t^k(da)\,dt$, since
$$
\langle N^{k,i}(B)\rangle_t = \int_0^t\Lambda_s(\pi_{n_k}^{-1}(B))\,ds = \int_0^t\Lambda_s^k(B)\,ds.
$$
Let $X^k$ solve
$$
dX_t^k =
\int_{A_{n_k}} b_{n_k}(t,X_t^k,\nu_t^{j,n_k},a)\Lambda_t^k(da)\,dt + 
\int_{A_{n_k}} \sigma_{n_k}(t,X_t^k,\nu_t^{j,n_k},a)N^k(da,dt),
\qquad X_0^k=X_0,
$$
and set $S^k:=P'\circ(\Lambda^k,X^k)^{-1}.$
Since $|\pi_{n_k}(a)|\le |a|$, the required control moment is finite,
and $X_0^k$ has law $\lambda$. Hence Proposition~\ref{proposition: martingale measure representation}, applied to the
truncated data, yields $S^k\in\mathcal R_{n_k}(\nu^{j,n_k}).$ 

We next prove that $S^k\to S$ in $\mathcal P^p(\Omega[A])$. The natural coupling of $\Lambda^k$ and $\Lambda$ gives
$$
\mathbb E^{P'}\left[d_{\mathcal V[A]}^p(\Lambda^k,\Lambda)\right] \le
\frac{1}{T}\,\mathbb E^{P'}\left[\int_0^T\int_A|\pi_{n_k}(a)-a|^p\Lambda_t(da)dt\right].
$$
Since $\pi_{n_k}(a)\to a$ and
$|\pi_{n_k}(a)|\le |a|$, the right-hand side converges to zero by dominated
convergence.

To treat the state processes, insert and subtract the coefficients evaluated
at $(X_t,\nu_t^j,a)$. Using the Lipschitz property in the state variable,
the Burkholder--Davis--Gundy inequality, and Young's inequality when $p<2$,
as in the stability estimate of \cite[Lemma~5.3]{lacker2015mean}, we obtain
$$
\mathbb E^{P'} \left [\|X^k-X\|_T^p \right ] \le C\int_0^T\mathbb E^{P'} \left [\|X^k-X\|_t^p \right ]dt + C\bigl(I_k^b+I_k^\sigma\bigr),
$$
where
$$
I_k^b := \mathbb E^{P'} \left [ \int_0^T\int_A \left|b_{n_k}(t,X_t,\nu_t^{j,n_k},\pi_{n_k}(a)) -
b(t,X_t,\nu_t^j,a)\right|^p \Lambda_t(da)dt \right ]
$$
and
$$
I_k^\sigma :=
\mathbb E^{P'}\left[\left(\int_0^T\int_A\left|\sigma_{n_k}(t,X_t,\nu_t^{j,n_k},\pi_{n_k}(a)) - \sigma(t,X_t,\nu_t^j,a)\right|^2\Lambda_t(da)dt\right)^{p/2}\right].
$$
By Lemma~\ref{lemma: time marginal continuity},
$\sup_{t\in[0,T]} W_p(\nu_t^{j,n_k},\nu_t^j) \rightarrow0$.
Together with $\pi_{n_k}(a)\to a$, continuity of the coefficients in
$(x,\mu,a)$, and convergence of the coefficient truncations, this gives
pointwise convergence of the integrands appearing in $I_k^b$ and
$I_k^\sigma$. The growth bounds in Assumption~\ref{ass_standing} and the $p'$-moment estimates,
with $p'>p\ge 1\vee p_\sigma$, provide the uniform integrability required in
the same argument as \cite[Lemma~5.3]{lacker2015mean}. Hence $I_k^b+I_k^\sigma\rightarrow0$.
Gronwall's lemma therefore yields
$$
\mathbb E^{P'} \left [\|X^k-X\|_T^p \right ]\rightarrow0.
$$
Combining this with the convergence of the projected controls gives $S^k\rightarrow S$
in $\mathcal P^p(\Omega[A])$.

It remains to prove convergence of the rewards. Since
$|\pi_{n_k}(a)|\le |a|$, the state estimate at order $p'$ and the preceding
moment bound give
$$
\sup_k \mathbb E^{P'}\left[\|X^k\|_T^{p'} + \int_0^T|\Lambda_t^k|^{p'}dt\right]<\infty.
$$
In addition, the projected control costs are dominated pathwise by the fixed integrable random variable associated with the original control:
$$
\int_0^T |\Lambda^k_t|^{p'} dt = \int_0^T \int_A |\pi_{n_k}(a)|^{p'} \Lambda_t(da)dt \leq \int_0^T |\Lambda_t|^{p'}dt.
$$
Continuity of $f$ and $g$, together with the convergence of the states,
controls, and environments, implies convergence in probability of the
corresponding reward variables. The state-dependent terms of order $p$ are
uniformly integrable because $p'>p$ and the states have uniformly bounded $p'$-moments; the deterministic environment terms are controlled by the uniform $p'$-moment bound on $(\nu^{j,n_k})_k$; and the only $p'$-order control term is dominated by the preceding integrable random variable. Hence both positive and negative parts of
the reward variables are uniformly integrable. Vitali's theorem therefore gives
$J(\nu^{j,n_k},S^k) \rightarrow J(\nu^j,S)$.
This completes the proof.
\end{proof}

\begin{proposition}[Approximation of admissible deviations in varying environments]
\label{proposition: deviation approximation varying environments app}
Let $\mu^j\rightarrow\mu$
in $\mathcal{P}^{p}(\mathcal{C}^d)$ and assume
$$
\sup_j \int_{\mathcal{C}^d} \|x\|_T^{p'}\,\mu^j(dx) <\infty.
$$
If $S\in\mathcal{R}(\mu)$ satisfies $J(\mu,S)>-\infty,$
then there exist $S^j\in\mathcal{R}(\mu^j)$ such that $S^j\rightarrow S$
in $\mathcal{P}^{p}(\Omega[A])$ and $J(\mu^j,S^j)\rightarrow J(\mu,S)$.
\end{proposition}

\begin{proof}
By lower semicontinuity and the uniform $p'$-moment bound on $(\mu^j)_j$,
we have $\|\mu\|_{T,p'}<\infty$. Hence
Lemma~\ref{lemma: finite value higher moments} yields
$$
\mathbb E^S\left[\|X\|_T^{p'}+ \int_0^T|\Lambda_t|^{p'}dt\right]<\infty.
$$
Realize $S$ on a filtered probability space
$(\Omega',\mathcal F',(\mathcal F'_t)_{t\in[0,T]},P')$
by $(\Lambda,X,N)$ as in
Proposition~\ref{proposition: martingale measure representation}, so that
$S=P'\circ(\Lambda,X)^{-1}$.
For each $j$, keep $\Lambda$ and $N$ fixed and solve
$$
dX_t^j=\int_A b(t,X_t^j,\mu_t^j,a)\Lambda_t(da)dt +\int_A\sigma(t,X_t^j,\mu_t^j,a)N(da,dt),
\qquad X_0^j=X_0.
$$
Then $S^j:=P' \circ (\Lambda,X^j)^{-1}\in\mathcal R(\mu^j)$.

The same BDG-Gronwall estimate used in the truncated approximation, now
without projection or truncation errors, gives
$$
\mathbb E^{P'} \left [\|X^j-X\|_T^p \right ]
\le C\int_0^T\mathbb E^{P'}\left [ \|X^j-X\|_t^p \right ]dt+C(B_j+\Sigma_j),
$$
where $B_j$ and $\Sigma_j$ contain only the differences between the
coefficients evaluated at $\mu_t^j$ and at $\mu_t$. By
Lemma~\ref{lemma: time marginal continuity},
$\sup_{t\in[0,T]}W_p(\mu_t^j,\mu_t)\rightarrow0$.
Continuity of the coefficients in the measure variable, together with the
growth bounds and the preceding $p'$-moment estimate, gives the required
uniform integrability, exactly as in the proof of the previous proposition.
Hence $B_j+\Sigma_j\to0$, and Gronwall's lemma yields $\mathbb E^{P'}[\|X^j-X\|_T^p]\rightarrow0$.
Since the control coordinate is unchanged, it follows that
$S^j\to S$ in $\mathcal P^p(\Omega[A])$.

Finally, the state estimate at order $p'$, the uniform $p'$-moment bound on
$(\mu^j)_j$, and the fixed $p'$-moment of $\Lambda$ give
$\sup_j\mathbb E^{P'} \left [\|X^j\|_T^{p'} \right ]<\infty$.
Here the control coordinate is unchanged, so the $p'$-order control-cost term is the same integrable random variable $\int_0^T |\Lambda_t|^{p'}dt$ for every $j$. Together with $p'>p$, the uniform $p'$-moment bounds for the states and environments therefore imply uniform integrability of both the positive and negative parts of the reward variables. Continuity of $f$ and $g$ then yields $J(\mu^j,S^j) \to J(\mu,S)$.
\end{proof}
\end{appendix}

\bibliographystyle{siam}
\bibliography{reference}
\end{document}